%% file: preprint.tex
\documentclass[11pt]{article}

\usepackage[a4paper,margin=1in]{geometry}

\usepackage[utf8]{inputenc}

\usepackage[ngerman,british]{babel}

\usepackage[T1]{fontenc}
\usepackage{mathpazo,tgcursor,tgpagella}
\usepackage{textcomp}

\usepackage[bookmarksnumbered,bookmarksopen,unicode]{hyperref}
\hypersetup{
	pdftitle={The Bipartite Boolean Quadric Polytope with Multiple-Choice Constraints},
	pdfauthor={Andreas Bärmann,Alexander Martin,Oskar Schneider},
	pdfkeywords={},
	pdfsubject={MSC2000 Primary ;Secondary} 
}

\usepackage{amsthm}

\usepackage{subcaption}
\usepackage{authblk}

\usepackage{algorithm}
\usepackage{algorithmicx}
\algnewcommand{\IIf}[1]{\State\algorithmicif\ #1\ \algorithmicthen}
\algnewcommand{\EndIIf}{\unskip\ \algorithmicend\ \algorithmicif}
\usepackage{mathdots}

\input{Packages}

\newtheorem{theorem}{Theorem}[section]
\newtheorem{corollary}[theorem]{Corollary}
\newtheorem{definition}[theorem]{Definition}

\newtheorem{proposition}[theorem]{Proposition}

\newtheorem{lemma}[theorem]{Lemma}

\input{Macros}
\newif\ifproof
\prooftrue 
\definecolor{lightmauve}{rgb}{0.86, 0.82, 1.0}

\DeclareSymbolFont{symbolsC}{U}{pxsyc}{m}{n}
\DeclareMathSymbol{\coloneqq}{\mathrel}{symbolsC}{"42}
\usepackage{cleveref}
\crefname{equation}{}{}
\renewcommand{\cref}{\Cref}

\title{The Bipartite Boolean Quadric Polytope\\with Multiple-Choice Constraints}

\author[1]{Andreas~Bärmann}
\author[1]{Alexander Martin}
\author[2]{Oskar Schneider\vspace{2\baselineskip}}
\affil[1]{
\url{Andreas.Baermann@math.uni-erlangen.de}

\url{Alexander.Martin@math.uni-erlangen.de}

Lehrstuhl für Wirtschaftsmathematik,

Department Mathematik,

Friedrich-Alexander-Universität Erlangen-Nürnberg,

Cauerstraße 11, 91058 Erlangen, Germany
\vspace{\baselineskip}
}
\affil[2]{
\url{Oskar.Schneider@fau.de}

Gruppe Optimization 

Fraunhofer Arbeitsgruppe für Supply-Chain Services SCS,

Fraunhofer Institut für Integrierte Schaltungen IIS,

Nordostpark 93, 90411 Nürnberg, Germany
}

\date{First Draft Online: 24 September 2020}

\begin{document}

\maketitle

\begin{abstract}
\input{abstract}
\end{abstract}

\input{introduction}
\input{properties}
\input{facets}
\input{separation}

\input{comp_results}
\input{conclusion}

\section*{Acknowledgements}
\label{sec:acknowledgements}

We thank Francisco Javier Zaragoza Mart\'{\i}nez for our fruitful discussions
on the topic as well as Mark Zuckerberg and Thomas Kalinowski
for clarifying details about Zuckerberg's technique for convex-hull proofs.
Futhermore, we acknowledge financial support
by the Bavarian Minis\-try of Economic Affairs, Regional Development and Energy
through the Center for Analytics -- Data -- Applications (ADA-Center)
within the framework of \qm{BAYERN DIGITAL~II}.

\bibliographystyle{alpha}
\bibliography{Literature}

\appendix
\input{appendix}

\end{document}

%% file: Packages.tex
\usepackage{amsmath,amssymb}
\usepackage{eurosym}
\usepackage{color}
\usepackage{eqparbox, array}
\usepackage{algpseudocode}
\newcommand*\Let[2]{\State #1 $\gets$ #2}
\newcommand*\Lett[4]{\State #1 $\gets$ #2, #3 $\gets$ #4}
\algrenewcommand\algorithmicrequire{\textbf{Input:}}
\algrenewcommand\algorithmicensure{\textbf{Output:}}

\usepackage{xspace}

\usepackage{float}

\usepackage{booktabs}

\usepackage{graphicx}
\usepackage{color}
\usepackage{epstopdf}
\DeclareGraphicsExtensions{.pdf,.eps,.png}

\usepackage{stackengine}
\usepackage{tikz}
\usetikzlibrary{arrows.meta,positioning}
\usetikzlibrary{shapes.misc}
\usetikzlibrary{shapes,arrows,chains}
\tikzstyle{line} = [draw, -latex']

\usepackage{chngcntr}
\counterwithout{equation}{section}

\usepackage{paralist}
\usepackage{enumitem}

\usepackage{bbm}

\usepackage{calc}

\usepackage{todonotes}
\presetkeys{todonotes}{inline}{}

\usepackage{numprint}
\npthousandsep{\,}

%% file: Macros.tex
\newcommand{\arrowone}{arrow-1}
\newcommand{\arrowtwo}{arrow-2}

\newcommand{\Arrowone}{Arrow-1}
\newcommand{\Arrowtwo}{Arrow-2}

\newcommand{\fancygraphname}{subset-uniform}

\renewcommand{\a}{a}
\newcommand{\e}{e}
\newcommand{\y}{y}
\newcommand{\Y}{Y}
\newcommand{\x}{x}
\newcommand{\X}{X}
\newcommand{\z}{z}
\newcommand{\I}{I}
\newcommand{\Is}{\mathcal{I}}
\renewcommand{\j}{j}

\renewcommand{\S}{S}
\newcommand{\h}{h}
\renewcommand{\H}{H}

\newcommand{\zP}{P}

\newcommand{\zsetforall}{\ \forall}

\newcommand{\F}{\{0, 1\}}
\newcommand{\FR}{[0, 1]}
\newcommand{\N}{\mathbbm{N}}

\newcommand{\R}{\mathbbm{R}}

\newcommand{\OO}{\mathcal{O}}

\newcommand{\SSet}[2]{%
		\left\{\, #1 \ \middle\vert \ #2 \, \right\}%
	}

\DeclareMathOperator{\conv}{conv}
\DeclareMathOperator{\supp}{supp}

\newcommand{\card}[1]{\lvert #1 \rvert}

\newcommand{\abs}[1]{\lvert #1 \rvert}

\newcommand{\set}[1]{\{#1\}}

\newcommand{\qm}[1]{``#1''}

\newcommand{\ie}{i.e.\ }
\newcommand{\Wlog}{w.l.o.g.\ }

\newcommand{\cf}{cf.\ }
\newcommand{\eg}{e.g.\ }

\newcommand{\mycomment}[1]{}

\newcommand{\st}{\mathrm{s.t.}}

\newlength\mysinglespace
\newlength\objspace
\newlength\conspace
\newlength\cconspace
\newenvironment{constr}[1]{\begin{array}[t]{#1}}{\end{array}} 

\newenvironment{opt*}[3]{\begin{equation*}\begin{array}{rl}#1 & #2 \\[\objspace] \st & \begin{constr}{#3}}{\end{constr}\end{array}\end{equation*}\\[0pt]}

\newenvironment{eqns*}[1]{\begin{equation*}\begin{array}[t]{#1}}{\end{array}\end{equation*}\\[0pt]}

\DeclareMathOperator{\argmax}{argmax}

\newcommand{\GG}{\mathcal{G}}
\newcommand{\EE}{\mathcal{E}}

\newcommand{\FF}{\mathcal{F}}
\newcommand{\LL}{\mathcal{L}}

%% file: abstract.tex
We consider the bipartite boolean quadric polytope (BQP)
with multiple-choice constraints and analyse its combinatorial properties.
The well-studied BQP is defined as the convex hull
of all quadric incidence vectors over a bipartite graph.
In this work, we study the case where there is a partition
on one of the two bipartite node sets
such that at most one node per subset of the partition can be chosen.
This polytope arises, for instance,
in pooling problems with fixed proportions
of the inputs at each pool.
We show that it inherits many characteristics from
BQP, among them a wide range of facet classes
and operations which are facet preserving.
Moreover, we characterize various cases
in which the polytope is completely described
via the relaxation-linearization inequalities.
The special structure
induced by the additional multiple-choice constraints
also allows for new facet-preserving symmetries
as well as lifting operations.
Furthermore, it leads to several novel facet classes
as well as extensions of these via lifting.
We additionally give computationally tractable exact separation algorithms,
most of which run in polynomial time.
Finally, we demonstrate the strength of both the inherited
and the new facet classes in computational experiments
on random as well as real-world problem instances.
It turns out that in many cases
we can close the optimality gap almost completely
via cutting planes alone, and, consequently,
solution times can be reduced significantly.

%% file: introduction.tex
\section{Introduction}

The \emph{boolean quadric polytope} of an undirected graph $ G = (V, E) $
is defined as
\begin{equation*}
	QP(G) \coloneqq \conv\SSet{(\x, \z)\in \F^{V\cup E}}{\x_i\x_\j=\z_{i\j} \ \forall \set{i,\j} \in E}
\end{equation*}
and was introduced by Padberg in \cite{padberg1989boolean}.
Due to its fundamental role in the field of polyhedral combinatorics
and its frequent occurrence in practical applications,
it has been extensively studied,
and many facet classes, corresponding separation algorithms,
symmetries and further geometric properties
have been found,
see \eg \cite{padberg1989boolean, sherali1995simultaneous, barahona1986cut, boros1993cut, letchford2014new}.
In \cite{de1990cut}, it has been shown
that~$ QP $ is the image of the \emph{cut polytope}
over an appropriate graph
under an affine transformation called \emph{covariance mapping}.
It has also been investigated under the name \emph{correlation polytope}
in \cite{pitowsky1991correlation}.
For an extensive compilation of the above results, and many more,
we refer the interested reader to \cite{deza1997geometry}.

In the special case where $ G = (\X \cup \Y, E) $ is a bipartite graph,
the polytope is called the \emph{bipartite boolean quadric polytope}:
\begin{equation*}
	BQP(G) \coloneqq \conv\SSet{(\x,\y,\z)\in \F^{\X\cup \Y\cup E}}{\x_i\y_\j=\z_{i\j} \ \forall \set{i,\j} \in E}.
\end{equation*}
and its geometry and further properties have been studied
in \cite{sripratak16bipartite, sripratak2014bipartite, PunnenSripratakKarapetyan2013, PunnenSripratakKarapetyan2015},
among others.
In the present article, we consider the bipartite case
together with an additional \emph{multiple-choice}
(or \emph{set-packing}) structure on the set $\X$.
Let $ \Is $
be a partition of $\X$,
and let $ \X^\Is \coloneqq
	\SSet{\x \in \F^{\X}}{\sum_{i \in \I} \x_i \leq 1
		\zsetforall \I \in \Is} $
be the set of incidence vectors
for which at most one entry per subset in the partition is set to one.
A substructure like this is prevalent in many applications
when competing choices or compatibilities between decision are involved.
This includes 
the knapsack problem
with multiple-choice constraints (\cite{Nauss1978,KellererPferschyPisinger2004}),
with applications in investment planning, among others,
or the clique problem
with multiple-choice constraints (\cite{staircase2018,cycle-free2020}),
which arises, for example, in scheduling problems (\cite{benders2020})
and flow problems with piecewise linear routing costs (\cite{liers2016structural}).
In the case of the boolean quadric polytope,
the multiple-choice structure models given, fixed proportions
between the multiplied quantities.
The underlying feasible set, our polytope of interest,
can then be stated as
\begin{equation*}
	\zP(G, \Is) \coloneqq \conv\SSet{(\x,\y,\z)\in \F^{\X \cup \Y \cup E}}{\x_i\y_\j=\z_{i\j} \ \forall \set{i,\j} \in E, x \in \X^\Is}.
\end{equation*}
Clearly, if all subsets in $\Is$ contain only one node each,
then $ P(G, \Is) = BQP(G) $,
otherwise it is obvious that $ P(G, \Is) \subseteq BQP(G) $,
but the former is not a face of the latter.
We will study structural properties of this polytope,
most notably symmetries, facet classes and separation routines,
some of which are inherited from the original (bipartite)~QP,
while others arise specifically due to the multiple-choice structure.

\paragraph{The boolean quadric polytope in bilinear programming}

The boolean quadric polytope and its variants play a major role
in the solution of bilinear programs. 
State-of-the-art solvers typically rely on linear programming (LP) relaxations
of the non-convex products of variables involved.
Most of them use the \emph{McCormick-relaxation}
(see \cite{mccormick1976computability}),
since it is the best possible linear relaxation
for the product of two continuous variables
over their (finite) bounds.
Obviously, valid inequalities for $ QP $ lead to improved relaxations.
If a bilinear program has further combinatorial substructures,
studying these in combination with $ QP $ allows for even tighter relaxations.
Examples of this are the boolean quadric polytope
over the forest sets of a graph (see \cite{lee2004boolean})
and the cardinality-constrained boolean quadric polytope
(see \cite{mehrotra1997cardinality, faye2005polyhedral, LimaGrossmann2017}).
In \cite{HardinLeeLeung1998},
the authors examine the boolean quadric packing
uncapacitated facility location polytope,
which models uncapacitated facility location with bilinear costs terms.
Further examples are \cite{Castro2015, BolandDeyKalinowskiMolinaroRigterink2015, fampa2018efficient, gupte2013solving},
\cf also \cite{KolodziejCastroGrossmann2013}
and the references therein.
In \emph{separable} (or \emph{disjoint}) bilinear programs,
the variables are partitioned into two subsets
such that there are no products
between any two variables in the same subset.
Moreover, these two subsets of variables
are not coupled via further (linear) constraints.
This special case motivates the study of the bipartite version of $ QP $, \ie $ BQP $,
see \eg \cite{galliseparable, gunluk2012polytope}.
In particular, the authors of \cite{gupteworking},
investigate the polyhedral structure
of separable bilinear programs
where one of the two variable sets obeys a single multiple-choice constraint.
For this case, they can give a complete description
of the convex hull of feasible solutions.
Furthermore, $ BQP $ with multiple-choice constraints is studied
in the context of the bipartite quadratic assignment problem
(see \eg \cite{PunnenWang2016}),
and, closely related, the bilinear assignment problem
(see \eg \cite{CusticSokolPunnenBhattacharya2017}).
Here, multiple-choice constraints can be used to model
the allocation of resources with different properties
such that precisely (or at most) one asset with a certain property
has to be positioned at each location.

\paragraph{Contribution and organization of the paper}
We are interested in bilinear programs where the variables in one of the subsets
are covered by non-overlapping multiple-choice constraints,
which leads us to study the polytope~$P$.
This work is motivated by a real-world pooling problem
(see \cite{misener2009advances, audet2004pooling} for a survey)
arising in the food industry.
There the products have to be manufactured according to given recipes
with fixed proportions of the ingredients.
For each of these ingredients,
there are potentially many different batches of different qualities
on stock to choose from.
The practical solution of pooling problems often relies on strong relaxations
of an underlying bilinear model,
as investigated in \cite{gupte2017relaxations, d2011valid}, for example.
In a similar spirit, we use our theoretical results
on the polyhedral structure of~$P$ to demonstrate their effect 
when solving the mentioned pooling problem with recipes.
We show that many of the original symmetries of $ BQP $
and further general characteristics
are preserved in~$P$, see \cref{sec:gen-prop}.
Moreover, a multitude of new facet-preserving symmetries arises,
which is very beneficial for the design of separation and lifting routines.
As we show in \cref{sec:fac-def-in},
we also inherit large part of the facial structure of $ BQP $.
However, the multiple-choice structure gives rise
to a variety of novel facet classes as well.
For the case of cycle-free dependencies
between the two bipartite subdivisions of the underlying graph,
we can even give a complete convex-hull description for~$P$.
In \cref{sec:sep}, we will devise separation algorithms
for several of the inherited and the new facet classes,
most of them running in polynomial time.
Finally, our computational experiments in \cref{sec:comp-res}
demonstrate that these routines are sufficient
to close the integrality significantly, sometimes completely.
Furthermore, we present a real-world computational study
for a pooling problem with recipes where we show
that the exploitation of the multiple-choice structure
outperforms a bilinear solver based on McCormick-relaxations
by orders of magnitude.
\Cref{sec:conclusion} rounds the paper off with our conclusions.
Finally, we give some of the details on the obtained results
in the appendix.

\paragraph{Notation}

For ease of notation, we denote for a vector $ a \in  \R^{X \cup Y \cup E} $
by $ a_i $ the component corresponding to node $ i \in X $,
similarly $ a_j $ for $ j \in Y $ and $ a_{ij} $ for $ \set{i, j} \in E $,
assuming an arbitrary fixed order on $ X \cup Y \cup E $.
Analogously, $ e_i $, $ e_j $ and $ e_{ij} $
denote the corresponding unit vectors.
For a node $ v \in X \cup Y $,
we define the \emph{neighbourhood} $ N(v) \coloneqq \set{w \in X \cup Y \colon \set{v, w} \in E} $.
Furthermore, we call the graph~$G$ \emph{\fancygraphname}
with respect to the partition~$ \Is $,
if any two nodes which are in the same subset of $I$
also have the same neighbourhood, \ie the same neighbours in $Y$.
For such a graph~$G$, we define
the corresponding \emph{dependency graph} $ \GG = (\Is \cup Y, \EE) $
by merging the nodes in each subset of the partition to a single node,
represented by the subset itself.
Its edge set $ \EE $ contains the merged edges $ \set{I, j} $
for all $ I \in \Is $ and all $ j \in J $ in the joint neighbourhood
of the original nodes in~$I$.
Valid inequalities $ a^T(x^T, y^T, z^T)^T \leq b $ for $ P(G, \Is) $
will be written as $ a^T(x, y, z) \leq b $ for short.
For such a valid inequality, we define the \emph{$Y$-support},
given by $ \supp_Y(a) \coloneqq \set{j \in Y \mid a_j \neq 0 \vee
	(\exists i \in N(j))\, a_{ij} \neq 0} $,
to denote the nodes in~$Y$ involved.
Finally, let $ [n] \coloneqq \set{1, \ldots, n} $.

%% file: properties.tex
\section{General properties of $\zP(G, \Is)$}
\label{sec:gen-prop}

We start with some basic properties of $ \zP(G, \Is) $.
In particular, these include symmetries of the polytope under different operations
on the coefficients of valid and facet-defining inequalities.
Some of these symmetries are inherited from $BQP$,
others are induced especially by our multiple-choice structure.

First, observe that it does not matter
whether we define $ \zP(G, \Is) $ as the convex hull
of binary or continuous vectors.
Namely, let $ \bar{X}^\Is \coloneqq \set{x \in [0, 1]^X \mid \sum_{i \in I} \x_i \leq 1 \zsetforall \I \in \Is} $,
and define the set $ S(G, \Is) \coloneqq \set{(x, y, z) \in \FR^{\X \cup \Y \cup E} \mid \x_i \y_\j = \z_{i\j} \zsetforall \set{i, \j} \in E,\, x \in \bar{X}^\Is} $.
Then it is obvious that the extreme points of $ \conv(S(G, \Is)) $
are of the form $ (\bar{x}, \bar{y}, \bar{z}) $,
where $ \bar{x} $ is a vertex of $ \bar{X}^\Is $,
$ \bar{y} $ is a vertex of $ [0, 1]^Y $
and $ \bar{z} = \bar{x} \bar{y} $.
This implies $ P(G, \Is) = \conv(S(G, \Is)) $,
and therefore $ \conv(S(G, \Is)) $ is a polytope.
For general properties of such polyhedra arising from separable bilinear programs
where the two underlying sets of feasible vectors are polytopes,
we refer to \cite{galliseparable}.

In the following, we give two immediate results about $ \zP(G, \Is) $.
\begin{proposition}[NP-hardness]
	Optimizing a linear objective over $ \zP(G, \Is) $ is NP-hard,
	even if each subset in the partition contains only one element.
\end{proposition}
\ifproof
\begin{proof}
	The NP-hardness of optimizing a linear objective over $ BQP(G) $,
	\ie the case with one element per subset,
	follows from \cite[Theorem~2.1]{sripratak2014bipartite}.
\end{proof}
\fi
\begin{proposition}[Dimensionality]
	The polytope $ \zP(G, \Is) $ is full-di\-men\-sional.
	\label{thm:dim}
\end{proposition}
\ifproof
\begin{proof}
	The polytope contains the $0$-vector
	as well as the vectors $ \e_{i} $
	for all $ i \in X$,
	$ \e_{\j} $ for all $ \j \in Y $
	and $ \e_{i} + \e_{\j} + \e_{i\j} $
	for all $ \set{i, \j} \in E $.
	These ($ \abs{\X} + \abs{\Y} + \abs{E} + 1 $)-many points
	are easily seen to be affinely independent.
\end{proof}
\fi
We will now focus on graph operations and symmetries
which preserve validity of an inequality for $ \zP(G, \Is) $
as well as the property to be facet-defining.
Let $ \hat{G} = (\hat{X} \cup \hat{Y}, \hat{E}) $ be a subgraph of $G$,
where $ \hat{X} \subseteq X $, $ \hat{Y} \subseteq Y $
and $ \hat{E} \subseteq \hat{X} \times \hat{Y} $.
Furthermore, let $ \hat{\Is} $ be the restriction
of the partition~$ \Is $
corresponding to $ \hat{X} $.
We define the \emph{extension}
of a valid inequality for $ \zP(\hat{G}, \hat{\Is})$
to a valid inequality for $ \zP(G, \Is) $
by adding a $0$-coefficient for each additional node and edge.
This procedure is called $0$-lifting in the literature.
Similarly, we define the \emph{extension}
of a point in $ \zP(\hat{G}, \hat{\Is}) $
to a point in $ \zP(G, \Is) $
by adding a $0$-entry for each additional node in~$X$,
which also uniquely defines the entries
corresponding to the additional edges.
The \emph{restriction} of a valid inequality for $ \zP(G, \Is) $
or a point in $ \zP(G, \Is) $ to $ \zP(\hat{G}, \hat{\Is}) $
is then defined conversely by discarding all components
for which there is no corresponding node or edge in $ \hat{G} $.
We obtain the following results
for extended and restricted inequalities respectively.
\begin{proposition}[Validity of extension and restriction]
	\label{prop:lifting_valid}\mbox{ }
	\begin{enumerate}
		\item If $ \hat{G} $ is a subgraph of $G$
			and $ \Is $ an extension of $ \hat{\Is} $,
			then the extension of a valid inequality
			for $ \zP(\hat{G}, \hat{\Is}) $
			is a valid inequality for $ \zP(G, \Is) $.
		\item If $ \hat{G} $ is an induced subgraph
			of $ \hat{G} $ and $ \hat{\Is} $ a restriction of $ \Is $,
			then the restriction of a valid inequality
			for $ \zP(G, \Is) $
			is a valid inequality for $ \zP(\hat{G}, \hat{\Is}) $.
	\end{enumerate}
\end{proposition}
\ifproof
\begin{proof}
1. The restriction of each point in $ \zP(G, \Is) $
is a point in $ \zP(\hat{G}, \hat{\Is}) $.
Therefore, the extended inequality is valid for $ \zP(G, \Is) $.
2. For each point in $ \zP(\hat{G}, \hat{\Is}) $,
the extension is a point in $ \zP(G, \Is) $.
As $ \hat{G} $ is an induced subgraph of $G$,
the extension only has additional $0$-entries.
Thus, the restricted inequality
is valid for $ \zP(\hat{G}, \hat{\Is}) $.
\end{proof}
\fi
\begin{proposition}[Facets by extension]
	\label{prop:facet-ext}
	Let $ \hat{G} = (\hat{X} \cup \hat{Y}, \hat{E}) $
	be an induced subgraph of~$G$
	such that $ \hat{X} = X $, and let $ \hat{\Is} = \Is $.
	Then the extension of any facet-defining inequality
	for $ \zP(\hat{G}, \Is) $ is also facet-defining for $ \zP(G, \Is) $.
\end{proposition}
\ifproof
\begin{proof}
	The result follows via the same construction
	of affinely independent points
	as in the proof for the original~$ QP $
	(see \cite[Theorem 3 (Lifting Theorem)]{padberg1989boolean}).
\end{proof}
\fi
\cref{prop:facet-ext} also holds
if we add a new subset to the partition which contains a single node.
However, inequality~\cref{rlt:3}, which we introduce in \cref{sec:basic_rlt},
is a counterexample which shows that increasing the graph
through an extension of the partition does generally not preserve facets.

Next, we consider symmetries of $ \zP(G, \Is) $
and their effects on its facial structure.
The following result for permutations
of constraint coefficients is straightforward.
\begin{proposition}[Permutation]
For a \fancygraphname\ graph~$G$, let $ \sigma $ be a permutation
on $ \X \cup \Y \cup E $
which permutes elements within a subset $ \I \in \Is $
or elements in~$Y$ and permutes the related edges in $E$ accordingly.
Then the following two statements are equivalent:
\begin{enumerate}
	\item The inequality $ a^T(x,y,z) \leq b $
		is valid (resp.\ facet inducing) for $ \zP(G,\Is) $.
	\item The inequality $ \sigma(a)^T(x,y,z) \leq b$
		is valid (resp.\ facet inducing) for $ \zP(G,\Is) $,
		where $ \sigma(a) $ denotes the resorted vector according to $ \sigma $.
\end{enumerate}
\label{prop:permu}
\end{proposition}
If $G$ is not \fancygraphname, \cref{prop:permu} only holds
when permuting nodes sharing the same neighbourhood
(together with their incident edges).
Furthermore, if~$I$ contains two subsets of the same size
in which the respective nodes share the same neighbourhood,
then the statement of \cref{prop:permu} also holds
for swapping all the nodes between these two subsets.

It is also possible to formulate a variant
of the well-known switching transformation
(see \cite[Theorem 6]{padberg1989boolean}) for $ \zP(G, \Is) $.
In contrast to the original transformation on~$ QP $,
it is, however, only possible to switch on a subset of the $y$-variables here.
\begin{proposition}[Switching]
	Let the inequality $ a^T(\x,\y,\z) \leq b $ be valid (resp.\ facet-inducing)
	for $ \zP(G, \Is) $, let further $ \hat{Y} \subseteq Y $ and define
	\begin{equation*}
	  \begin{split}
		\hat{a}_i &\coloneqq
		\begin{cases}
			a_i + \sum_{\j \in \hat{Y}} a_{i\j} & \text{for }\I \in \Is, i \in \I\\
		\end{cases} \\
			\hat{a}_\j &\coloneqq
		\begin{cases}
			-a_\j & \text{for } \j \in \hat{Y}\\
			a_\j & \text{otherwise }
		\end{cases}\\
	\end{split}
	\quad
	\begin{split}
			\hat{a}_{ij} &\coloneqq
		\begin{cases}
			-a_{ij} & \text{for } \set{i,\j} \in E \text{ with } \j \in \hat{Y}\\
			a_{ij} & \text{otherwise }
		\end{cases}\\
		\hat{b} &\coloneqq b - \sum_{\j \in \hat{Y}} a_\j
	\end{split}
	\end{equation*}
	then $ \hat{a}^T(\x, \y, \z) \leq \hat{b} $
	is a valid (resp.\ facet-inducing) inequality for $ \zP(G, \Is) $ as well.
	\label{lemma:sym}
\end{proposition}
\begin{proof}
	Define the one-to-one mapping $ \psi\colon \zP(G, \Is) \to \zP(G, \Is) $
	for which $ v = (x, y, z) $ is mapped onto $ (\hat{x}, \hat{y}, \hat{z}) $
	with $ \hat{x} = x $ as well as
	\begin{equation*}
	\hat{y}_j \coloneqq
	\begin{cases}
	1 - \y_\j  & \text{for } \j \in \hat{Y} \\
	\y_\j & \text{otherwise }\\
	\end{cases} \quad \text{and} \quad
		\hat{\z}_{i\j} \coloneqq
	\begin{cases}
	- \z_{i\j} + \x_i  & \text{for } \set{i, \j} \in E \text{ with } \j \in \hat{Y} \\
	z_{i\j} & \text{otherwise }\\
	\end{cases}.
	\end{equation*}
	One can easily check that $ v_1, \ldots, v_m \in \zP(G, \Is) $
	are affinely independent
	iff $ \psi(v_1), \ldots, \psi(v_m) \in \zP(G, \Is) $ are.
	Thus, $ \psi $ maps facets onto facets.
\end{proof}
Note that the $Y$-support of an inequality is preserved under switching.
Moreover, one can separate in polynomial time
over all switchings of a given inequality
by iteratively checking if the inequality is tightened
by switching on each $ j \in Y $ separately.

The following symmetric operation on $ \zP(G, \Is) $,
which we call \emph{copying},
is novel in the sense that there is no corresponding operation
on either~$ QP $ or~$ BQP $.
It arises specifically due to the multiple-choice constraints.
\begin{definition}[Copying]
	Let $ a^T(\x, \y, \z) \leq b $ be a valid inequality for $ \zP(G, \Is) $.
	For an $ \I \in \Is $ and an $ i \in \I $
	let $ r_i \in \R^{1 + \card{N(i)}} $ be the tuple of the coefficients of $a$
	corresponding to the node~$i$ and the incident edges.
	Further, let $ H^I $ be the set of coefficient tuples
	for all nodes in a subset~$I$.
	For a \fancygraphname\ graph, a \emph{copying} of~$a$
	is obtained by replacing each coefficient tuple~$ r_i $
	for $ I \in \Is $, $ i \in I $, by some tuple in $ h \in H^{I} $.
	If we further restrict the copying operation
	such that in the transformed inequality
	each element in $ H^I \setminus \set{(0, 0, \ldots, 0)} $
	for each $ I \in \Is $ needs to be chosen at least once,
	then we call it a \emph{structure-preserving} copying.
	\label{def:subset_lifting_validness}
\end{definition}
\begin{proposition}[Validity of copying]
	Let $ a^T(x, y, z) \leq b $ be a valid inequality for $ \zP(G, \Is) $,
	then any copied inequality is also valid.
	\label{lemma:subset_lifting_validness_ok}
\end{proposition}
\begin{proof}
	Due to $ x \in \X^\Is $,
	only one variable per subset in~$ \Is $ can be set to~$1$
	in a feasible solution.
	The correctness then follows from \cref{prop:permu}.
\end{proof}
If the graph is not \fancygraphname,
\cref{lemma:subset_lifting_validness_ok} still holds
if the copying is only performed among nodes
which share the same neighbourhood.

Similar to switching, copying is able to generate an exponential number
of new valid inequalities from some given valid inequality.
Note that one can separate over all of these copyings in polynomial time
by iteratively checking which coefficient tuple makes the inequality tightest
for each element of each subset of the partition separately.
Thus, copying provides a very efficient lifting procedure.
An interesting question is now
if copying maps facets onto facets.
We will see in~\cref{sec:basic_rlt}
that in general the answer is no,
where the inequality \cref{rlt:1} can serve as a counterexample.
However, in \cref{sec:cycle} we will introduce
several classes of facet-defining inequalities
for which structure-preserving copying maps
facets onto other facets.
For the remainder of this paper,
we will always mean structure-preserving copying when we refer to copying.

%% file: facets.tex
\section{Facet-defining inequalities}
\label{sec:fac-def-in}

In this section, we will describe several classes
of facet-defining inequalities for the polytope $ \zP(G, \Is) $
and characterize cases
in which they are sufficient to completely describe its convex hull.

\subsection{Basic and RLT inequalities}
\label{sec:basic_rlt}

The following valid inequalities for $ \zP(G, \Is) $
are part of its definition,
which is why we call them the \emph{basic inequalities}:
\begin{alignat}{1}
	0 \leq \y_{\j} &\leq 1 \quad \forall \j \in Y, \label{basic:1}\\
	\x_{i} &\geq 0 \quad \forall \I \in \Is, i \in \I, \label{basic:2}\\
	\sum_{i \in \I} \x_i &\leq 1 \quad \forall \I \in \Is \label{basic:3}.
\intertext{
By applying the well-known \emph{Reformulation-Linearization Technique}
(see \cite{sherali1992new})
to the inequalities \cref{basic:1}--\cref{basic:3},
we obtain the following system of inequalities:}
z_{i\j} &\geq 0 \quad \forall \set{i,\j} \in E \label{rlt:1},\\ 
\x_{i} -z_{i\j} &\geq 0  \quad \forall \set{i,\j} \in E, \label{rlt:2}\\
\y_{\j} - \sum_{i \in N(j) \cap I} z_{i\j} &\geq 0 \quad \forall \I \in \Is, \j \in Y,  \label{rlt:3}\\
 \y_{\j} +\sum_{i \in N(j) \cap I} \left(\x_i - z_{i\j} \right) &\leq  1\quad \forall \I \in \Is, \j \in Y. \label{rlt:4}
\end{alignat}
We will call \cref{rlt:1}--\cref{rlt:4} the \emph{RLT inequalities}.
In some cases, the basic inequalities
are already facet-defining by themselves;
however, they are most of the time dominated by the RLT inequalities,
as the following two results show.
\begin{theorem}[Basic facets]\mbox{ }
	\begin{enumerate}
		\item Iff $ N(j) = \emptyset $
			for some $ \j \in Y $,
			the bounds in~\cref{basic:1}
			define facets of $ \zP(G, \Is) $.
		\item Iff $ N(i) = \emptyset $
			for some $ i \in \X $,
			inequality~\cref{basic:2}
			defines a facet	of $ \zP(G, \Is) $.
		\item Iff $ \bigcap_{i \in I} N(i) = \emptyset $
			for some $ \I \in \Is $,
			inequality~\cref{basic:3}
			defines a facet of $ \zP(G, \Is) $. 
	\end{enumerate}
	\label{thm:char_basic}
\end{theorem}
\ifproof
\begin{proof}
	It easy to see that the stated conditions
	are sufficient for the corresponding inequalities to be facet-defining.
	Otherwise, the lower bound in \cref{basic:1}
	is dominated by \cref{rlt:1} and \cref{rlt:3},
	the upper bound in \cref{basic:1}
	is dominated by \cref{rlt:2} and \cref{rlt:4},
	\cref{basic:2} is dominated by \cref{rlt:1} and \cref{rlt:2},
	and \cref{basic:3} is dominated by \cref{rlt:3} and \cref{rlt:4}.
\end{proof}
\fi
\begin{theorem}[RLT facets]
	The RLT inequalities \cref{rlt:1}--\cref{rlt:4}
	define facets of the polytope $ \zP(G, \Is) $,
	except for the case when $ N(j) \cap I = \emptyset $
	in \cref{rlt:3} or \cref{rlt:4}.
\end{theorem}
\ifproof
\begin{proof}
	The $0$-vector as well as $ \e_{i} $ for $ i \in X $,
	$ \e_{\j} $ for $ \j \in Y $ and $ \e_{i} + \e_{\j} + \e_{i\j} $
	for $ \set{i, \j} \in E \setminus \set{i_1,\j_1} $
	satisfy $ z_{i_1\j_1} = 0 $ for each $ \set{i_1, \j_1} \in E $.
	They are affinely independent,
	which shows that \cref{rlt:1} defines a facet.
	Inequality \cref{rlt:2} induces a facet
	since it is a switching of \cref{rlt:1}
	for $ \hat{Y} \coloneqq \set{j_1} $.
	The $0$-vector as well as $ \e_{i} $ for $ i \in X $,
	$ \e_{\j} $ for $ \j \in Y \setminus \set{j_1} $
	and $ \e_{i} + \e_{\j} + \e_{i\j} $ for $ \set{i, \j} \in E $
	satisfy $ \y_{\j_1} - \sum_{i \in N(j_1) \cap \I_1} z_{i\j_1} = 0 $
	for each $ \I_1 \in \Is $ and $ \j_1 \in Y $.
	Their affine independence shows that \cref{rlt:3}
	defines a facet as well.
	Again, \cref{rlt:4} is a switching of \cref{rlt:3}
	for $ \hat{Y} \coloneqq \set{j_1} $ and thus also induces a facet.
\end{proof}
\fi

\subsubsection{Complete description on cycle-free dependency graphs}
\label{sec:compl-descpr}
In the following, we describe sufficient conditions
for the graph~$G$ and the partition~$ \Is $
which guarantee that the basic and RLT inequalities
completely describe $ P(G, \Is) $.

An initial, simple conclusion can be drawn directly
from Theorem~1 in \cite{gupteworking}:
in case there is only one subset in the partition~$ \Is $,
\ie $ \Is = \set{X} $, and~$G$ is a complete bipartite graph,
the RLT inequalities are indeed sufficient to describe $ P(G, \Is) $.
We will now generalize this finding in two ways:
our main result will be that the basic and RLT inequalities are sufficient
for subset-uniform graphs~$G$ with a cycle-free dependency graph,
independent from the number of subsets in the partition.
Then we will see that the basic and RLT inequalities together
fully describe $ P(G, \Is) $ for arbitrary bipartite graphs~$G$
if, as in \cite{gupteworking}, $ \Is = \set{X} $ holds.

Our proofs are based on Zuckerberg's method for deriving convex-hull descriptions
for combinatorial problems
(see \cite{zuckerberg2016geometric, bienstock2004subset}).
We briefly summarize it here,
based on the simplified formulation given in \cite{gupte2020extended}.
Consider a $0$/$1$-polytope $ R \coloneqq \conv(\FF) $
with vertex set $ \FF \subseteq \F^n $
and a second polytope $ H \subseteq \R^n $
(typically given via an inequality description)
for which we would like to show $ R = H $.
We can prove this by verifying both $ \FF \subseteq H $ and $ H \subseteq R $.
To show the latter inclusion,
we first need to represent $ \FF $
as a finite set-theoretic expression
consisting of unions, intersections and complements
of the sets
\begin{equation*}
	A_i \coloneqq \SSet{a \in \F^n}{a_i = 1}, \quad i = 1, \ldots, n.
	\label{equ:zuck-a}
\end{equation*}
Let $ F(A_1, \ldots, A_n) $ be such a representation.
Further, define $ U \coloneqq [0, 1) $,
let $ \LL $ be the set of all unions
of finitely many half-open subintervals of $U$,
and let $ \mu $ be the Lebesgue measure (restricted to $ \LL $), that is
\begin{alignat*}{1}
	\LL \coloneqq \SSet{[a_1, b_1) \cup \ldots \cup [a_k,b_k)}
{0 \leq a_1 < b_1 < a_2 < b_2 < \ldots < a_k < b_k \leq 1, k \in \N},\\[0.25\baselineskip]
	\mu(S) \coloneqq (b_1 - a_1) + \ldots + (b_k - a_k)
\text{ for any } S = [a_1, b_1) \cup \ldots \cup [a_k, b_k) \in \LL.
\end{alignat*}
For each point $ \h \in \H $, we then have to find
a collection of sets $ S_1, \ldots, S_n \in \LL $
such that $ \mu(\S_i) = \h_i $ for all $ i \in [n] $
and $ F(\S_1,\ldots, \S_n) = U $
(where the complement is taken in~$U$ instead of $ \F^n $).
With these prerequisites, the following theorem, a slight reformulation
of \cite[Theorem~4]{gupte2020extended},
gives the desired result.
\begin{theorem}[Zuckerberg's convex hull characterization]
	Let $ \FF \subseteq \F^n $, $ h \in \FR^n $,
	and let~$F$ be a finite set theoretic expression
	with $ F(A_1, \ldots, A_n) = \FF $.
	Then $ h \in \conv(\FF) $
	if and only if there are sets $ S_1, \ldots, S_n \in \LL $
	such that $ \mu(S_i) = h_i $ for all $ i \in [n] $,
	and $ F(S_1, \ldots, S_n) = U $.
	\label{thm:gupte}
\end{theorem}
\begin{proof}
	Combine Theorem~4 with the arguments in Remark~2,
	both from \cite{gupte2020extended}.
\end{proof}
A detailed discussion of this proof technique
can be found in the mentioned literature.

It is obvious that every binary set $ \FF $
can be represented by some formula~$F$
by encoding each point~$ v \in \FF $ separately,
namely via the set-theoretic subexpressions $ F_v(A_1, \ldots, A_n) \coloneqq
	\bigcap_{i \in [n]:\, v_i = 1} A_i \cap
		\bigcap_{i \in [n]:\, v_i = 0} \bar{A_i} = \set{v} $.
Forming the union of these subexpressions
yields a representation~$ F(A_1, \ldots, A_n) $ of $ \FF $
in disjunctive normal form.
However, from this form it is usually very hard
to deduce a construction rule for the sets $ S_1, \ldots, S_n $
fulfilling all requirements.
A~compact description of~$F$ is much more indicative in this respect,
such as the one we give for $ P(G, \Is) $ in the following.
\begin{lemma}[Set characterization for $P(G, \Is)$]
	Let $ \FF \coloneqq P(G, \Is) \cap \F^{X \cup Y \cup E} $
	and $ h \in \FR^{X \cup Y \cup E} $.
	Then $ \h \in \conv(\FF) $
	if and only if there are sets $ \S_{i} \in \LL $ for $ i \in X $,
	$ \S_{\j} \in \LL $ for $ \j \in Y $
	and $ \S_{i\j} \in \LL $ for $ \set{i, j} \in E $
	satisfying all of the following conditions:
	\begin{enumerate}[label=(\roman*)]
		\item $ \mu(\S_i) = \h_i $ for all $ i \in X $,
		\item $ \mu(\S_\j) = \h_\j $ for all $ \j \in Y $,
		\item $ \mu(\S_{i\j}) = \h_{i\j} $ for all $ \set{i, \j} \in E $,
		\item $ \S_i \cap \S_\j = \S_{i\j} $ for all $ \set{i, \j} \in E $, \label{enum:set_char4}
		\item $ \S_{i_1} \cap \S_{i_2} = \emptyset $
			for all distinct $ i_1, i_2 \in \I $ with $ \I \in \Is $.
	\end{enumerate}
	\label{zuckerberg:p}
\end{lemma}
\begin{proof}
	Let $ A_i $ for $ i \in X $, $ A_j $ for $ j \in J $
	and $ A_{ij} $ for $ \set{i, j} \in E $
	be defined analogously to~\cref{equ:zuck-a}.
	The set $ \FF $ can then be written as
	\begin{equation*}
	 	F(A_1, \ldots, A_n) \coloneqq
	 		\bigcap_{\set{i, \j} \in E} \underbrace{
	 			(A_i \cap A_\j \Leftrightarrow  A_{i\j})}_{(a)\, x_i y_i = z_{ij}}
	 			\cap \bigcap_{\I \in \Is} \underbrace{\overline{
	 				\bigcup_{\substack{i_1, i_2 \in \I:	i_1 \neq i_2}}
	 					\underbrace{A_{i_1} \cap A_{i_2}}
	 						_{x_{i_1} + x_{i_2} \leq 1}}}
	 							_{(b)\, \sum_{i \in I} x_{i} \leq 1}.
	\end{equation*}
	Each intersection in the above expression
	represents a defining constraint
	for the vertices of $ \zP(G, \Is) $.
	Subexpressions~(a) ensure that all resulting vectors
	are valid for $ BQP(G) $
	while subexpressions~(b) require them
	to fulfil the multiple-choice constraints.
	Conditions (iv)	and (v) together, stemming from (a) and~(b) respectively,
	are now equivalent
	to $ F(\S_1, \ldots, \S_n) = U $.
	The rest follows from \cref{thm:gupte}.
\end{proof}

We readily see that \cref{zuckerberg:p}
also yields a certificate for $ h \in BQP(G) $
(or $ h \in QP(G) $ for a non-bipartite graph~$G$)
if we omit condition~(v).
We only need to define~$F$
without the intersection with subexpressions~(b) in the proof.

We are now ready to prove our main convex-hull result.
\begin{theorem}[Complete description for cycle-free dependency graphs]
	Let $H$ be the polytope defined by the basic
	and RLT inequalities \cref{basic:1,basic:2,basic:3,rlt:1,rlt:2,rlt:3,rlt:4}.
	If $G$ is a \fancygraphname\ graph
	and the dependency graph $ \GG $ is cycle-free,
	we have $ H = \zP(G, \Is) $.
\label{thm:compl_subsetuniform}
\end{theorem}
\begin{proof}
	First, note that $ \zP(G, \Is) \subseteq H $ holds,
	as \cref{basic:1,basic:2,basic:3,rlt:1,rlt:2,rlt:3,rlt:4}
	are valid constraints for $ \zP(G, \Is) $.
	To prove $ H \subseteq \zP(G, \Is) $,
	we give an explicit construction
	of the sets $ S_i $, $ S_j $ and $ S_{ij} $
	for each point $ h \in H $
	as required in \cref{zuckerberg:p}.
	As $ \GG $ is cycle-free,
	we can represent it as a finite list of trees.
	This allows us to define the required sets
	via the algorithm described in the following,
	which recursively traverses each tree.
	It consists of the three routines
	\textsc{Define-Sets}, \textsc{Traverse-Tree} and \textsc{Match}.
	The auxiliary routine \textsc{Match} is given in \cref{fig:def_match}.
	Its inputs are a set $ S \in \LL $
	together with a list of diameters $ (w_1, \ldots, w_k) $ for some $ k \geq 1 $.
	The output is then a list of pairwise disjoint subsets $ (S_1, \ldots, S_k) $ of~$S$
	with $ \mu(S_i) = w_i $ for all $ i \in [k] $.
	This is possible precisely
	if the stated requirements for the diameters~$ w_i $ are fulfilled.
	\begin{figure}[htb]
		\hfill
		\begin{subfigure}[b]{0.58\textwidth}
			\begin{algorithmic}[1]
				\Require{$ S \in \LL,
					(w_1, \ldots, w_k), w_i \in [0, 1), i \in [k] $
					with $ w_1 + \ldots + w_k \leq \mu(S) $}
				\Ensure{$ (S_1, \ldots, S_k) $
					with $ S_i \subseteq S $ for $ i \in [k] $,
					$ S_i \cap S_j = \emptyset $ for $ i, j \in [k] $
					with $ i \neq j $}
				\Function{Match}{$ S, (w_1, \ldots, w_k) $}
					\State $ t_0 \leftarrow 0 $
					\For{$ r = 1, \ldots, k $}
						\State $ t_r \leftarrow
							\min\set{t \in S \mid \mu(S \cap [t_{r - 1}, t]) = w_r} $
						\State $ S_r \leftarrow S \cap [t_{r - 1}, t_r) $
					\EndFor
					\State{\textbf{return} $ (S_1, \ldots, S_k) $}
				\EndFunction
			\end{algorithmic}
		\end{subfigure}
		\hfill
		\begin{subfigure}[b]{0.35\textwidth}
			\begin{tikzpicture}[xscale=6,yscale=1]
			\draw[thick] (0,0) -- (0.5,0);
			\draw[thick] (0,.1) -- (0,.-.1);
			\draw (.1,.05) -- (.1,-.05);
			\draw (.2,.05) -- (.2,-.05);
			\draw (.3,.05) -- (.3,-.05);
			\draw (.4,.05) -- (.4,-.05);
			\draw[thick] (.5,.1) -- (.5,.-.1);
			\node at (0,-.4) {$0$};
			\node at (.5,-.4) {$1$};
			\draw[fill=red!30] (0,0.2) rectangle ++(0.1,0.4);
			\draw[fill=red!30] (0.15,0.2) rectangle ++(0.15,0.4);
			\draw[fill=red!30] (0.4,0.2) rectangle ++(0.1,0.4);
			\draw[fill=blue!30] (0,0.2+.5*1) rectangle ++(0.05,0.4);
			\draw[fill=green!30] (0.05,0.2+.5*2) rectangle ++(0.05,0.4);
			\draw[fill=green!30] (0.15,0.2+.5*2) rectangle ++(0.1,0.4);
			\draw[fill=yellow!30] (0.25,0.2+.5*3) rectangle ++(0.05,0.4);
			\draw[fill=yellow!30] (0.4,0.2+.5*3) rectangle ++(0.05,0.4);
			\node at (-.1,.4+.5*0) {$\S$};
			\node at (-.1,.4+.5*1) {$\S_1$};
			\node at (-.1,.4+.5*2) {$\S_2$};
			\node at (-.1,.4+.5*3) {$\S_3$};
			\end{tikzpicture}
			\hfill\null
		\end{subfigure}
		\caption{Subroutine \textsc{Match} (left)
			and exemplary output for defining three subsets of some set~$S$ (right)}
		\label{fig:def_match}
	\end{figure}
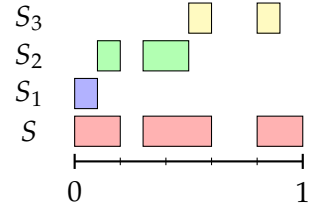

	For the remaining two routines, which are shown in \cref{fig:zuckerberg_in_action},
	we consider the graph~$G$, its dependency graph~$ \GG $
	and the (arbitrary) point $ h \in H $ whose membership in $ \zP(G, \Is) $
	shall be verified as global variables, \ie they are known everywhere.
	The same holds for the sets required for \cref{zuckerberg:p},
	which become known globally once defined via the symbol $ \coloneqq $.
	In contrast, we define auxiliary sets,
	which are local to a routine,
	via the symbol $ \leftarrow $.
	\begin{figure}[p]
		\null\vfill\null
		\begin{algorithmic}[1]
			\Function{Define-Subsets}{}
			\For{each tree in $ \GG $}
			\State Pick some $ \j \in Y $ as the root node \Wlog
			\State $ \S_j \coloneqq [0, \h_j) $.
			\For{$ c \in N(j) $}
			\State \textsc{Traverse-Tree}($j$, $c$)
			\EndFor
			\EndFor
			\EndFunction	
			\Function{Traverse-Tree}{$p$, $c$}
			\If{$ c \in \Is $}
			\State Let $ (i_1, \ldots, i_k) $
			be any fixed order of the nodes in $c$
			\State $ (\S_{i_1 p}, \ldots, \S_{i_k p}) \coloneqq
			\textsc{Match}(S_p, (h_{i_1 p}, \ldots, h_{i_k p})) $
			\Comment{M1, sets for the edges}
			\State $ (\hat{\S}_{i_1}, \ldots, \hat{\S}_{i_k}) \leftarrow
			\textsc{Match}(\overline{\S_p},
			(h_{i_1} - h_{i_1 p}, \ldots,
			h_{i_k} - h_{i_k p})) $
			\Comment{M2}
			\State $ (\S_{i_1}, \ldots, \S_{i_k}) \coloneqq
			(\hat{\S}_{i_1} \cup \S_{i_1 p}, \ldots,
			\hat{\S}_{i_k} \cup \S_{i_k p}) $
			\Comment{sets for the nodes}
			\Else \Comment{$ c \in Y $}
			\For{$ i \in p $}
			\State $ \S_{i c} \coloneqq
			\textsc{Match}(\S_i, (h_{i c })) $
			\Comment{M3, sets for the edges}
			\EndFor
			\State $ \hat{\S}_c \leftarrow
			\textsc{Match}(\overline{\cup_{i \in p}\S_{i}},
			(h_{c} - \sum_{i \in p}h_{i c})) $
			\Comment{M4}
			\State $ \S_c \coloneqq
			  \bigcup_{i \in p}  \S_{i c}  \cup \hat{\S}_c$
			\Comment{set for the node $c$}
			\EndIf
			\For{$ r \in N(c) $}
			\State \textsc{Traverse-Tree}($c$, $r$)
			\EndFor
			\EndFunction
		\end{algorithmic}
		\null\vfill\null
		\centering
		\begin{tikzpicture}[xscale=20,yscale=1]
		\draw[thick] (0,0) -- (0.5,0);
		\draw[thick] (0,.1) -- (0,.-.1);
		\draw (.1,.05) -- (.1,-.05);
		\draw (.2,.05) -- (.2,-.05);
		\draw (.3,.05) -- (.3,-.05);
		\draw (.4,.05) -- (.4,-.05);
		\draw[thick] (.5,.1) -- (.5,.-.1);
		\node at (0,-.4) {$0$};
		\node at (.5,-.4) {$1$};
		\draw[fill=red!30] (0,0.2) rectangle ++(0.3,0.4);
		\draw[fill=blue!30] (0,0.2+.5*1) rectangle ++(0.1,0.4);
		\draw[fill=green!30] (0.1,0.2+.5*2) rectangle ++(0.1,0.4);
		\draw[fill=yellow!30] (0,0.2+.5*3) rectangle ++(0.1,0.4);
		\draw[fill=yellow!30] (0.3,0.2+.5*3) rectangle ++(0.05,0.4);
		\draw[fill=orange!30] (0.1,0.2+.5*4) rectangle ++(0.1,0.4);
		\draw[fill=orange!30] (0.35,0.2+.5*4) rectangle ++(0.05,0.4);
		\draw[fill=black!30] (0,0.2+.5*5) rectangle ++(0.1,0.4);
		\draw[fill=black!30] (0.3,0.2+.5*5) rectangle ++(0.03,0.4);
		\draw[fill=purple!30] (0.1,0.2+.5*6) rectangle ++(0.1,0.4);
		\draw[fill=purple!30] (0.35,0.2+.5*6) rectangle ++(0.03,0.4);
		\draw[fill=cyan!30] (0,0.2+.5*7) rectangle ++(0.33,0.4);
		\draw[fill=cyan!30] (0.35,0.2+.5*7) rectangle ++(0.03,0.4);
		\draw[fill=cyan!30] (0.4,0.2+.5*7) rectangle ++(0.05,0.4);
		\node at (-.1,.4+.5*0) {$S_{j_1}$};
		\node at (-.1,.4+.5*1) {$S_{i_1j_1}$};
		\node at (-.1,.4+.5*2) {$S_{i_2j_1}$};
		\node at (-.1,.4+.5*3) {$S_{i_1}$};
		\node at (-.1,.4+.5*4) {$S_{i_2}$};
		\node at (-.1,.4+.5*5) {$S_{i_1j_2}$};
		\node at (-.1,.4+.5*6) {$S_{i_2j_2}$};
		\node at (-.1,.4+.5*7) {$S_{j_2}$};
		\end{tikzpicture}  
		\caption{Routines \textsc{Define-Sets} and \textsc{Traverse-Tree} (top)
			and exemplary construction of the intervals
			for the complete bipartite graph~$G$
			with $ X = \set{i_1, i_2} $ and $ Y = \set{j_1, j_2} $
			as well as the partition $ \Is = {X} $ (bottom).
			The dependency graph~$ \GG $ of~$G$ has node set~$ \Is $
			and edges $ \EE = \set{\set{\Is, j_1}, \set{\Is, j_2}} $.
			The algorithm starts by defining the set $ S_{j_1} $,
			constructs the intermediate sets for edges
			as well as the nodes in $ \Is $
			and terminates after the definition of $ S_{j_2} $.
			The so-defined sets satisfy the conditions in \cref{zuckerberg:p},
			especially~(iv) and~(v),
			where in this case the former reads $ S_{i_1} \cap S_{i_2} = \emptyset $
			while the latter is $ S_{i_q} \cap S_{j_p} = S_{i_q j_p} $
			for all $ q = 1, 2 $ and $ p = 1, 2 $.}
		\label{fig:zuckerberg_in_action}
		\null\vfill\null
	\end{figure}
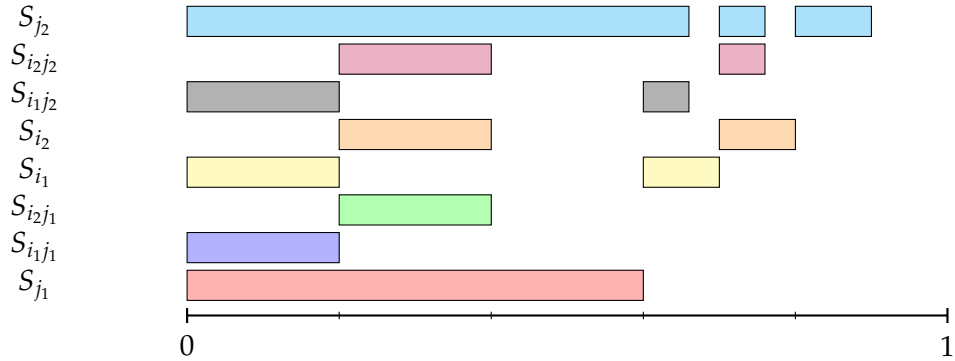
	
	Function \textsc{Define-Sets} is the main routine of the algorithm.
	It iterates over each tree in $ \GG $
	and defines the corresponding sets for the nodes and edges independently,
	which is possible due to \cref{zuckerberg:p}.
	For the current tree, the first step
	is to choose an arbitrary root node $ j \in Y $
	and to define the corresponding
	set as $ S_j \coloneqq [0, h_j) $.
	Then we perform depth-first search by calling \textsc{Traverse-Tree}
	for all edges $ \set{j, c} $ of~$ \GG $ with $ c \in N(j) $.
	This subroutine has two input parameters:
	a parent node~$p$ and a child node~$c$ (relative to the chosen root node).
	It assumes that the sets for~$p$ have already been defined
	and constructs the sets for the node~$c$
	as well as for the edges between~$c$ and~$p$.
	To define these sets, we need several calls to the function \textsc{Match}
	and need to argue each time why its input requirements are satisfied
	for $ h \in H $.
	Clearly, inequalities~\cref{rlt:1} ensure that the inputs are non-negative.
	The function starts by differentiating two cases,
	namely $ c \in \Is $ or $ c \in Y $.
	
	In the first case, it fixes some order of the nodes in~$c$
	and performs Match~M1, which directly defines the sets
	for the edges between the nodes in~$c$ and~$p$.
	Note that inequalities \cref{rlt:3}, which hold for~$h$,
	ensure that the requirements for~M1 are satisfied.
	From \cref{zuckerberg:p}, condition~(iv),
	we know that for any node $ i \in c $ connected to node~$p$
	the corresponding sets have to fulfil $ S_i \cap S_p = S_{i p} $.
	As $ S_{i p} $ has already been defined via~M1,
	part of the set $ S_i $ is thereby already known.
	The other part is given by the auxiliary set $ \hat{S}_{i_l} $,
	which is determined by Match~M2	over the complement of $ S_p $,
	where inequality~\cref{rlt:4} ensures its requirements.
	This way, all sets for the nodes in~$c$
	as well as their connecting edges to~$p$ have been defined
	and we can go further down the tree.
		
	For the second case, $ c \in Y $,
	\textsc{Traverse-Tree} performs Match~M3 in a loop to define the sets
	for the edges between $p$ and $c$,
	by which the set $ S_c $ is already partly determined,
	similar as in the first if-branch.
	The remaining part of set $ S_c $ is finally defined by Match~M4
	within the complement of $ \cup_{i \in p}  \S_{i c} $.
	Matches~M3 and~M4 are possible due to \cref{rlt:3} and \cref{rlt:4}.
	Again, all required sets corresponding to $p$, $c$
	and the edges in between them have been determined
	and we can continue recursively with~$c$ and its children.
	
	Whenever we come across an isolated node $ j \in J $,
	\cref{basic:1} ensures $ S_j \subseteq [0, 1) $.
	For isolated nodes $ i \in I $ for some $ I \in \Is $,
	the \fancygraphname ity of the graph
	implies that all other nodes in~$I$ are isolated as well.
	Thus, the sets corresponding to the nodes in~$I$
	can simply be placed next to each other without overlap,
	and inequalities~\cref{basic:2} and~\cref{basic:3} together
	ensure $ \sum_{i \in \I} \mu(S_i) \leq 1 $.
	
	By construction, all the defined sets
	satisfy the requirements of \cref{zuckerberg:p},
	which finishes the proof.
\end{proof}
The above proof also yields an alternative way
to show \cite[Proposition~8]{padberg1989boolean},
which states that the RLT inequalities
are sufficient to completely describe $ QP $ on cycle-free graphs.
Furthermore, the following corollary gives another case
where the RLT inequalities suffice to completely describe $ \zP(G, \Is) $.
\begin{corollary}[Complete description on graphs with one subset]
	Let $G$ be an arbitrary bipartite graph, and let $ I = \set{X} $,
	then the polytope $H$ as defined in \cref{thm:compl_subsetuniform}
	fulfils $ H = \zP(G, \Is) $.
	\label{cor:compl_onesubset} 
\end{corollary}
The proof of \cref{cor:compl_onesubset}
is similar to that of \cref{thm:compl_subsetuniform}.
The sets for~$X$ are chosen adjacent to each other, starting from~$0$.
Then for each $ \set{i, j} \in E $, $ h_{ij} $
is matched onto the set~$ S_i $.
Finally, for $ j \in Y $ the surplus $ h_j - \sum_{\set{i, j} \in E} h_{ij} $
is matched onto $ \overline{\cup_{\set{i, j} \in E} S_i} $.

\subsection{Lifted facets from BQP facets}

In this section, we describe classes of facets of $ \zP(G, \Is) $
which are inherited from $ BQP $.
By applying the switching operation from \cref{prop:lifting_valid}
and the copy operation from \cref{lemma:subset_lifting_validness_ok},
we will be able to lift them,
which enables us to produce large classes
of new facets as well.

\subsubsection{Cycle inequalities}
\label{sec:cycle}

A well-known class of facets for $ QP $
are the cycle inequalities,
which were introduced in \cite{padberg1989boolean}.
The $0$-lifting of a \emph{basic cycle inequality} to $ P(G, \Is) $
can be stated as:
\begin{equation}
	-z_{i_1 j_1} + z_{i_1 j_m} + \sum_{p = 2}^m
		\left(-\y_{j_p} - x_{i_p} + z_{i_p j_{p - 1}} + z_{i_p j_p}\right) \leq 0,
\label{eq:cycle_general_class_short}
\end{equation}
for each cycle $ \set{\set{i_1, j_1}, \set{j_1, i_2}, \set{i_2, j_2}, \ldots,
	\set{i_m, j_m}, \set{j_m, i_1}} \subseteq E $
of length $ 2m $ in $ G $ for some~$m$,
where $ i_1, \ldots, i_m $ are from different subsets of the partition~$ \Is $.
Note that only even cycles are possible,
because~$G$ is a bipartite graph.
Furthermore, to obtain facets we only need to consider cycles
which touch at most one node per subset.
Otherwise, \cref{eq:cycle_general_class_short}
is the sum of multiple cycle inequalities fulfilling this property.

The basic cycle inequalities~\cref{eq:cycle_general_class_short}
together will all inequalities obtained from them via switchings
are commonly subsumed under the name \emph{cycle inequalities},
a notion which we adopt as well.
The new copy operation further enlarges this facet class
to include those induced by
\begin{equation}
	\sum_{i \in S_1} \left(-z_{i j_1} + z_{i j_m} \right)
		+ \sum_{p = 2}^m \left(-\y_{j_p}
		+ \sum_{i \in S_p}
			\left(-x_i + z_{i j_{p - 1}} + z_{i j_p}\right)\right) \leq 0,
\label{eq:cycle_general_class_short_copy}
\end{equation}
for each simple, chordless cycle $ (\set{I_1, j_1}, \set{j_1, I_2}, \set{I_2, j_2}, \ldots,
	\set{I_m, j_m}, \set{j_m, I_1}) \subseteq \hat{E} $ in $ \GG $
of size~$ 2m $ for some~$m$
and for all non-empty subsets $ S_1 \subseteq \I_1, \ldots, S_m \subseteq \I_m $.
These inequalities define facets if the cycle is chordless,
as we will see in~\cref{Thm:Cycle+Copying_Facets},
otherwise they can be split into two shorter cycle inequalities.
We refer to inequalities~\cref{eq:cycle_general_class_short_copy}
as well as all their switchings as the \emph{cycle+copying inequalities}.
In \cref{fig:cycle}, the support of these inequalities
is also shown in matrix form, a notation we adopt from \cite{collins2004relevant}.
The values in the second column and the second row of the matrix
are the coefficients of the involved $x$- and $y$-variables respectively,
and the lower right matrix contains the coefficients
of the corresponding $z$-variables.
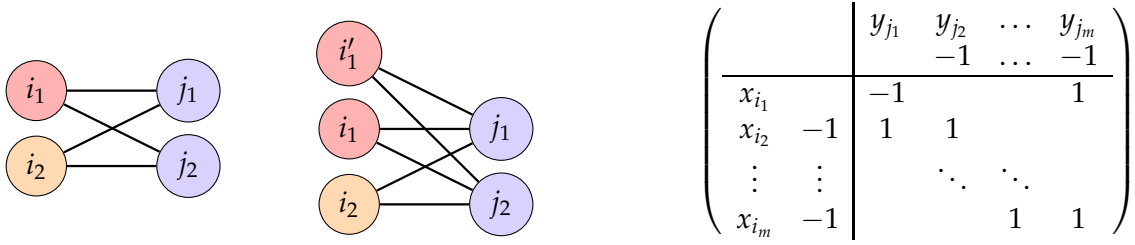
\begin{figure}
	\null\hfill\null
	\begin{subfigure}[b]{0.25\textwidth}
		\begin{tikzpicture}[baseline=(current bounding box.center),xscale=1,yscale=1]
			\node[shape=circle,draw=black, fill=red!30] (B) at (0,1) {$i_1$};
			\node[shape=circle,draw=black, fill=orange!30] (A) at (0,0) {$i_2$};
			\node[shape=circle,draw=black, fill=lightmauve] (D) at (2,0) {$j_2$};
			\node[shape=circle,draw=black, fill=lightmauve] (E) at (2,1) {$j_1$};
			\path [-, line width=0.3mm, draw=black] (B) edge node[left] {} (E);
			\path [-, line width=0.3mm, draw=black] (B) edge node[left] {} (D);
			\path [-, line width=0.3mm, draw=black] (A) edge node[left] {} (D);
			\path [-, line width=0.3mm, draw=black] (A) edge node[left] {} (E);
		\end{tikzpicture}
	\end{subfigure}
	\null\hfill\null
	\begin{subfigure}[b]{0.25\textwidth}
		\begin{tikzpicture}[baseline=(current bounding box.center),xscale=1,yscale=1]
			\node[shape=circle,draw=black, fill=red!30] (B) at (0,1) {$i_1$};
			\node[shape=circle,draw=black, fill=red!30] (C) at (0,2) {$i'_1$};
			\node[shape=circle,draw=black, fill=orange!30] (A) at (0,0) {$i_2$};
			\node[shape=circle,draw=black, fill=lightmauve] (D) at (2,0) {$j_2$};
			\node[shape=circle,draw=black, fill=lightmauve] (E) at (2,1) {$j_1$};
			
			\path [-, line width=0.3mm, draw=black] (C) edge node[left] {} (D);
			\path [-, line width=0.3mm, draw=black] (C) edge node[left] {} (E);
			\path [-, line width=0.3mm, draw=black] (B) edge node[left] {} (E);
			\path [-, line width=0.3mm, draw=black] (B) edge node[left] {} (D);
			\path [-, line width=0.3mm, draw=black] (A) edge node[left] {} (D);
			\path [-, line width=0.3mm, draw=black] (A) edge node[left] {} (E);
		\end{tikzpicture}
	\end{subfigure}
	\null\hfill\null
	\begin{subfigure}[b]{0.48\textwidth}
		\begin{equation*}
		\left(\begin{array}[c]{cc|ccccc}
		&    & \y_{j_1} & \y_{j_2} & \ldots &\y_{j_m} \\
		&    &       &   -1    &  \ldots  & -1 \\ \hline
		x_{i_1} &  &   -1   &      &       & 1\\
		x_{i_2} &  -1  &   1    &    1  &     &  \\
		\vdots & \vdots &       &   \ddots &\ddots          & \\
		x_{i_m} &  -1  &      &      & 1   & 1
		\end{array}\right)
		\end{equation*}
	\end{subfigure}
	\null\hfill\null
	\caption{Support graph of the shortest possible
		cycle inequality~\cref{eq:cycle_general_class_short}
		induced by four nodes $ \set{i_1, j_1, i_2, j_2} $
		with $ i_1, i_2 \in X $ from different subsets in $ \Is $
		and $ j_1, j_2 \in Y $ (left).
		Support graph of the shortest possible
		cycle+copying inequality~\cref{eq:cycle_general_class_short_copy}
		induced by the four nodes from before and an additional node
		$ i'_1 \in X $ from the same subset as $ i_1 $ (middle).
		The support of a basic cycle inequality in matrix form (right).}
	\label{fig:cycle}
\end{figure}
	
\begin{theorem}[Cycle+copying facets]
	The cycle+copying inequalities~\cref{eq:cycle_general_class_short_copy}
	define facets for $ \zP(G, \Is) $ if the underlying cycle is chordless.
	\label{Thm:Cycle+Copying_Facets}
\end{theorem}
\begin{proof}
	Let the subset $ C \coloneqq (\set{I_1, j_1}, \set{j_1, I_2}, \set{I_2, j_2}, \ldots,
		\set{I_m, j_m},\set{j_m, I_1}) \subseteq \hat{E} $
	be a chord\-less cycle of size~$ 2m $ in $ \GG $ for some $ m \geq 2 $,
	and let the subsets $ S_1 \subseteq \I_1, \ldots, S_m \subseteq \I_m $
	be non-empty.
	If $ \card{S_1} = \ldots = \card{S_m} = 1 $,
	\cref{eq:cycle_general_class_short_copy}
	is the basic cycle inequality for $ BQP(G) $,
	and its validity can be inferred from~\cref{prop:lifting_valid}.
	If $ \card{S_p} > 1 $ for some $ p \in [m] $,
	validity follows from \cref{lemma:subset_lifting_validness_ok}.
	
	Now, let $ a^T(x, y, z) \leq b $
	be a facet-defining inequality for $ \zP(G, \Is) $
	which contains the face~$F$ induced by \cref{eq:cycle_general_class_short_copy}.
	We then choose arbitrary representatives $ i_1 \in I_1, \ldots, i_m \in I_m $
	and show that~$a$ and~$b$ are multiples
	of the coefficients of inequality~\cref{eq:cycle_general_class_short_copy}
	by constructing $ \card{X} + \card{Y} + \card{E} $
	affinely independent points on~$F$ via \cref{Alg:Cycle_Facets}.
	\begin{algorithm}[h]
		\begin{algorithmic}[1]
			\State From $ 0 \in F $ follows $ b = 0 $
			\State From $ \e_{j_1} \in F $ follows $ \a_{j_1} = 0 $
			\State From $\e_{i_1} \in F$ follows $ \a_{i_1} = 0 $
			\State Let $ u \leftarrow e_{j_1} $
			\For{$ h = 2, \ldots, m $}
				\State $ u \leftarrow u + \e_{i_h} + \e_{i_h j_{h - 1}} $
					\Comment{$u$ could not be on $F$ if the cycle had a chord}
				\State From $ u \in F $ follows $ \a_{i_h} = -a_{i_h j_{h - 1}} $
				\State $ u \leftarrow u + \e_{j_h} +  \e_{i_h j_h} $
				\State From $ u \in F $ follows $ \a_{j_h} = -a_{i_h j_h} $
			\EndFor
			\State Let $ v \leftarrow \e_{i_1} + \e_{j_m} + \e_{i_1 j_m} $
			\State From $ v \in F $ follows $ \a_{j_m} = -\a_{i_1 j_m} $
			\For{$ h = m, \ldots, 2 $}
				\State $ v \leftarrow v + \e_{i_h} +  \e_{i_h j_h} $
				\State From $ v \in F $ follows $ \a_{i_h} = -a_{i_h j_h} $
				\State $ v \leftarrow w + \e_{j_{h - 1}} + \e_{i_h j_{h - 1}} $
				\If{$ h > 2 $}
					\State From $ v \in F $
						follows $ \a_{j_{h - 1}} = -a_{i_h j_{h - 1}} $
				\EndIf
			\EndFor
			\State From $ v + \e_{i_1 j_1} \in F $
				follows $ \a_{i_1 j_1} = -\a_{i_2 j_1} $
			\State Let $ \hat{Y} \leftarrow Y \setminus \set{j_1, \ldots, j_m} $
				and $ \hat{I} = \Is \setminus \set{I_1, \ldots, I_m} $
			\State From $ \e_j \in F $ follows $ \a_j = 0 $ for all $ j \in \hat{Y} $
			\State From $ \e_i \in F $ follows $ \a_i = 0 $
				for all $ i \in I $ with $ I \in \hat{I} $
			\State From $ \e_j + \e_i + \e_{ij} \in F $
				follows $ \a_{ij} = 0 $ for all $ \set{i, j} \in E $
					with $ i \in I $, $ I \in \hat{I}, j \in \hat{Y} $
			\State From $ \e_{j_1} + \e_i + \e_{ij_1} \in F $
				follows $ \a_{i j_1} = 0 $ for all $ \set{i, j_1} \in E $
				with $ i \in I $, $ I \in \hat{I} $
			\State From $ \e_j + \e_{i_1} + \e_{i_1 j} \in F $
				follows $ \a_{i_1 j} = 0 $
				for all $ \set{i_1, j} \in E $ with $ j \in \hat{Y} $
			\State Let $ w \leftarrow e_{j_1} $
			\For{$ h = 2, \ldots, m $}
				\State $ w \leftarrow w + \e_{i_h} + \e_{i_h j_{h - 1}} $
				\State From $ w + \e_j + \e_{i_h j} \in F $
					follows $ \e_{i_h j} = 0 $ for all $ \set{i_h, j} \in E $
					with $ j \in \hat{Y} $
				\State $ w \leftarrow w + \e_{j_h} + \e_{i_h j_h} $
				\State From $ w + \e_i + \e_{ij_h} \in F $
					follows $ \e_{ij_h} = 0 $
					for all $ \set{i, j_h} \in E $ with $ i \in I $, $ I \in \hat{I} $
			\EndFor
		\end{algorithmic}
		\caption{Construction of affinely independent points on a cycle facet}
		\label{Alg:Cycle_Facets}
	\end{algorithm}
	
	The copy operation produces different possibilities
	for the remaining coefficients.
	Indeed, for each $ i \in I_r $, $ r = 1, \ldots, m $,
	there are two possibilities:
	either all coefficients associated with~$i$
	(for the $x$- and $z$-variables) are $0$,
	then similar steps as 22--34 are necessary,
	or they are identical to the coefficients of the $ i_r $,
	which can be seen by replacing~$i$ with~$i_r$ in \cref{Alg:Cycle_Facets}.
	This proves the claim.
\end{proof}
We remark that if the graph is not \fancygraphname,
inequality~\cref{eq:cycle_general_class_short_copy}
can be defined accordingly in terms of cycles of the original graph,
and the sets $ S_1, \ldots, S_m $ may only contain nodes with the same neighbourhood.

For the facet classes of $ \zP(G, \Is) $ we derive in the following,
we assume~$G$ to be a complete bipartite graph
to allow for a simpler presentation. 

\subsubsection{$ I_{mm22} $ Bell inequalities}
\label{sec:bell}

A second prominent class of valid inequalities for $ BQP $
are the $ I_{mm22} $ Bell inequalities,
see \eg \cite{avis2007new, collins2004relevant, werner2001bell}.
Their $0$-lifted version can be stated as
\begin{equation}
	- \y_{\j_1}  
		- \sum^m_{k = 1} (m - k) \x_{i_k} 
		- \sum_{\substack{2 \leq p, k \leq m:\\p + k = m + 2}} z_{i_p j_k} 
		+ \sum_{\substack{1 \leq p, k \leq m:\\p + k < m + 2}} z_{i_p j_k} \leq 0,
	\label{equ:bell}
\end{equation}
for pairwise distinct $ I_1, \ldots, I_m \in \Is $,
$ i_1 \in I_1, \ldots, i_m \in I_m $
as well as pairwise distinct $ j_1, \ldots, j_m \in Y $ for some $ m \geq 2 $,
\cf \cref{Fig:Support_Bell}.
These inequalities induce facets as well.
\begin{figure}[h]
	\begin{equation*}
		\left(\begin{array}{cc|cccc}
					&   		& \y_{j_1}	& \y_{j_2}	& \ldots	& \y_{j_m}\\
					&    		& -1		&			&			& \\\hline
			x_{i_1} & -(m - 1)	& 1			& 1			& \ldots   	& 1\\
			x_{i_2} & -(m - 2)  & 1			& \iddots	& \iddots   & -1\\
			\vdots	& \vdots 	& \vdots    & \iddots	& \iddots   &\\
			x_{i_m} & 0 		& 1  		& -1		&			& 
			\end{array}\right)
	\end{equation*}
	\caption{Support of the $ I_{mm22} $ Bell inequalities}
	\label{Fig:Support_Bell}
\end{figure}
\begin{theorem}[$ I_{mm22} $ Bell facets]
	Inequalities~\cref{equ:bell} define facets of $ \zP(G, \Is) $
	if~$G$ is a complete bipartite graph.
	\label{thm:bell}
\end{theorem}
\begin{proof}
	From \cref{cor:compl_onesubset},
	we know that inequalities~\cref{equ:bell} define facets
	if all subsets of the partition contain only one element.
	To prove the claim via induction,
	we show that if we increase one subset $ \I \in \Is $
	by one new node~$i$,
	the corresponding extension of this inequality still defines a facet.
	Let $F$ be the face defined by this extended inequality.	
	The extensions of the vertices of the original facet
	also lie on~$F$.
	Thus, we only need to construct $ (\card{Y} + 1) $-many
	additional affinely independent points on~$F$.
	We can start by choosing $ \e_i \in F $
	and $ \e_i + \e_j + \e_{ij} \in F $ for $ j \in Y \setminus \set{j_1} $.
	In case $ \I \neq \I_m $,
	we can next use $ \e_i + \e_{i_m} + \e_{j_1} + \e_{i j_1} + \e_{i_mj_1} \in F $.
	Otherwise, we can select $ \e_i + \e_{i_{m - 1}}
		+ e_{j_2} + \e_{j_1} + \e_{i j_1} + \e_{i_{m - 1} j_1}
		+ \e_{i j_2} + \e_{i_{m - 1} j_2} \in F $,
	concluding the proof.
\end{proof}
The above proof also works if the support
of the $ I_{mm22} $ Bell inequality to be lifted
is a complete bipartite subgraph of~$G$.

Note that we have only considered one of the two possible valid inequalities
for $ \zP(G, \Is) $ which could be derived
from the original $ I_{mm22} $ Bell inequalities for $ BQP $.
A similar proof as above also works for swapped coefficients
of~$x$ and~$y$ in inequality~\cref{equ:bell}.

\paragraph{Further facets from BQP}

We have shown exemplarily for two classes of facet-defining inequalities for~$ BQP $
that they can be $0$-lifted to produce facets for $ \zP(G, \Is) $.
For the Bell inequalities,
the corresponding proof requires only a simple extension
of the original proof for $ BQP $.
Furthermore, we have seen that the copy operation
allows us to significantly enlarge a given inequality class.
All the resulting inequalities can be facet-defining,
as was the case for the cycle inequalities.
However, proving this fact was much more involved,
as the necessary affinely independent points on the facet
needed to be constructed from scratch. 
Our computational experiments with facet enumeration tools
like \emph{polymake} (see \cite{polymake:2000})
for small instances hint that there may be many more facet classes
which $ P(G, \Is) $ inherits from $ BQP $,
and their variants obtained via copying
could be facet-defining as well.

In \cite{avis2005two}, the authors introduce a technique
called \emph{triangular elimination},
which transforms a facet of the cut polytope on an appropriate graph
to a facet of $ BQP $ defined on the complete bipartite graph
(see \cite{avis2008generating} for the non-complete case). 
The cut polytope is very well investigated,
and there are many known facet classes for it,
like the hypermetric, clique-web and parachute inequalities and many more
(see \cite{deza1997geometry}). 
The triangular eliminations of these inequalities
give rise to a rich pool of facet classes for $ BQP $,
for which it could be tested as well
if they are $0$-liftable to facets for $ \zP(G, \Is) $
and if their copyings define facets, too.

\subsection{Novel facets for $ \zP(G, \Is) $}

Besides the facets inherited from $ BQP $, the polytope $ \zP(G, \Is) $
has a large variety of facets specifically induced
by the multiple-choice constraints.
In the following, we introduce a very rich template
of cutting planes for $ \zP(G, \Is) $,
which we call the $ 1, m $-inequalities.
As two special cases of this inequality,
we present the facet-inducing \arrowone{} and \arrowtwo{} inequalities,
whose switchings and copyings are facet-defining as well.

\subsubsection{$ 1, m $-inequalities}

We define the class of \emph{$ 1, m $-inequalities} as
\begin{equation}
	a_{i_1} x_{i_1} 
		- \sum_{p = 1}^k a_{j_p} \y_{j_p} 
		+ \sum_{p = 1}^m a_{i_1j_p}  z_{i_1j_p} 
		+ \sum_{h = 2}^m \sum_{p = 1}^m  a_{i_h j_p} z_{i_h j_p} \geq 0,
	\label{equ:new-valid-inq1}
\end{equation}
with distinct $ I_1, I_2 \in \Is $, $ i_1 \in I_1 $,
pairwise distinct $ i_2, \ldots, i_m \in I_2 $
as well as pairwise distinct $ j_1, \ldots, j_m \in Y $ for some $ m \geq 3 $
and $ 1 \leq k \leq m$.
The following conditions
guarantee their validity for $ P(G, \Is) $.
\begin{theorem}
	Inequality \cref{equ:new-valid-inq1} is valid and supporting for $ P(G, \Is) $
	if its coefficients satisfy all of the following conditions:
	\begin{enumerate}[label=(\roman*)]
		\item $ a_{i_1 j_p} = -1, \quad \forall p = 1, \ldots, m $
		\item $ a_{j_p} = 1, \quad \forall p = 1, \ldots, k $
		\item $ a_{i_1} = m - k $
		\item $ a_{i_h j_p} \in \set{0, -1},
				\quad \forall p = 1, \ldots, k, \quad \forall h = 2, \ldots, m$
		\item $ a_{i_h j_p} \in \set{0, 1},
				\quad \forall p = 1, \ldots, k, \quad \forall h = 2, \ldots, m $
		\item $ \sum_{p = 1}^k a_{i_h j_p}
				+ \sum_{p \in R} \left(a_{i_h j_p} - 1 \right) \geq -(m - k),
				\quad \forall h = 2, \ldots, m,
				\quad \forall R \subseteq \set{k + 1, \ldots, m} $
	\end{enumerate}
	\label{thm:1minequalities}
\end{theorem}
\begin{proof}
	Let $r$ be a vertex of $ P(G, \Is) $.
	We distinguish the following four exhaustive cases:
	for $ r_{i_1} = 0 $ and $ r_{i_h} = 0 $ for all $ h = 2, \ldots, m $,
	the inequality is trivially satisfied.
	For $ r_{i_1} = 1 $ and $ r_{i_h} = 0 $ for all $ h = 2, \ldots, m $,
	conditions~(i)--(iii) ensure validity.
	For $ r_{i_1} = 0 $ and $ r_{i_h} = 1 $
	for some $ h \in \set{2, \ldots, m} $,
	it is ensured by conditions~(ii), (iv) and (v).
	For $ r_{i_1} = 1 $ and $ r_{i_h} = 1 $
	for some $ h \in \set{2, \ldots, m} $,
	the inequality is valid by conditions~(i)--(vi).
	Furthermore, the inequality is always binding for the $0$-vector. 
\end{proof}

We now present two special cases
in which the $ 1, m $-inequalities define facets.

\subsubsection{\Arrowone{} inequalities}
The first class of facets is defined by
what we call the \emph{\arrowone{} inequalities}:
\begin{equation}
	(m - 1) x_{i_1} + \y_{j_1} - \sum_{p = 1}^m z_{i_1 j_p}
		+ \sum_{p = 2}^m \left(-z_{i_p j_1} + z_{i_p j_m} \right) \geq 0,
	\label{equ:new_facet1}
\end{equation}
for distinct $ I_1, I_2 \in \Is $, $ i_1 \in I_1 $,
pairwise distinct $ i_2, \ldots, i_m \in I_2 $
and pairwise distinct $ j_1, \ldots, j_m \in Y $
with $ m \geq 3 $.
Their support is depicted in \cref{fig:supp_arrow1} on the left.
As these inequalities arise from the $ 1,m $-inequalities
by setting $ k = 1 $, we know from \cref{thm:1minequalities}
that they are valid for $ \zP(G, \Is) $.
The next result shows that they are even facet-defining.
\begin{figure}[t]
	\null\hfill\null
	\begin{subfigure}[b]{0.5\textwidth}
		\begin{equation*}
			\left(\begin{array}{cc|cccc}
							&			& \y_{j_1}	& \y_{j_2}	& \ldots	& \y_{j_m}\\
							&			& 1			&			&			&\\\hline
				\x_{i_1}	& m - 1		& -1  		& -1		& \ldots	& -1\\
				\x_{i_2}	&			& -1		& 1			&			&\\
				\vdots		&			& \vdots	&			& \ddots	&\\
				\x_{i_m}	&			& -1		&			&			& 1
			\end{array}\right)
		\end{equation*}
	\end{subfigure}
	\null\hfill\null
	\begin{subfigure}[b]{0.45\textwidth}
		\begin{equation*}
			\left(\begin{array}{cc|cccc}
						&	& \y_{j_1}	& \y_{j_2}	& \ldots	& \y_{j_m}\\
						&	&			& 1			& \ldots	& 1\\\hline
				x_{i_1}	& 1	& -1		& -1		& \ldots	& -1\\
				x_{i_2} &	& 1			& -1		&			&\\
				\vdots	&	& \vdots	&			& \ddots	&\\
				x_{i_m} &	& 1			&			&			& -1
			\end{array}\right)
		\end{equation*}
	\end{subfigure}
	\null\hfill\null
	\caption{Support of the \arrowone (left) and \arrowtwo (right) inequalities}
	\label{fig:supp_arrow1}
\end{figure}
\begin{theorem}[\Arrowone{} facets]
	Inequality~\cref{equ:new_facet1}
	as well as all its copyings define facets of $ \zP(G, \Is) $
	if~$G$ is a complete bipartite graph.
	\label{Thm:Arrow-1_Facets}
\end{theorem}
\begin{proof}
	We fix some distinct $ I_1, I_2 \in \Is $, a $ i_1 \in I_1 $
	as well as pairwise distinct $ i_2, \ldots, i_m \in I_2 $
	and pairwise distinct $ j_1, \ldots, j_m \in Y $
	for some $ m \geq 3 $.
	Let $ a^T(x, y, z) \leq b $ be a facet-defining inequality
	which contains the face~$F$ induced by \cref{equ:new_facet1}.
	In~\cref{Alg:Arrow-1_Facets}, we indicate how to construct
	$ \card{X} + \card{Y} + \card{E} $ affinely independent points on $F$	
	to show that $a$ and~$b$ are multiples of the coefficients of \cref{equ:new_facet1}.
	\begin{algorithm}
		\begin{algorithmic}[1]
			\State From $ 0 \in F $ follows $ b = 0 $
			\State From $ \e_{j_p} \in F$ follows $ \a_{j_p} = 0 $
				for all $ p = 2, \ldots, m $
			\State From $ \e_{i_h} \in F $ follows $ \a_{i_h} = 0 $
				for all $ h = 2, \ldots, m $
			\State From $ \e_{j_1} + \e_{i_h} + \e_{i_h j_1} \in F $
				follows $ \a_{i_h j_1} = -\a_{j_1} $ for all $ h = 2, \ldots, m $
			\State From $ \e_{j_1} + \e_{j_p}
				+ \e_{i_h} + \e_{i_h j_1} + \e_{i_h j_p} \in F $
				follows $ \a_{i_h j_p} = 0 $ for all $ h = 2, \ldots, m $
				and $ p = 2, \ldots, m $ with $ h \neq p $
			\State From $ \e_{j_1} + \e_{j_p}
				+ \e_{i_p} + \e_{i_l}
				+ \e_{i_pj_1} + \e_{i_lj_1} + \e_{i_pj_p} + \e_{i_lj_p} \in F $
				for some arbitrary $ l \in \set{2, \ldots, m} \setminus p $
				follows $ \a_{i_p j_p} = \a_{j_1} $ for all $ p = 2, \ldots, m $
			\State Let $ v \leftarrow \e_{j_1} + \e_{i_1}
				+ \e_{i_1 j_1} + \sum_{h = 2}^m (\e_{i_h} + \e_{i_hj_1}) $
			\State From $ v \in F $
				and $ v + \e_{j_p} + \e_{i_1 j_p} + \sum_{r = 2}^m \e_{i_r j_p} \in F $
				follows $ \a_{i_1 j_p} = -\a_{j_1} $ for all $ p = 2, \ldots, m $
			\State From $ \e_{i_1} + \sum_{h = 2}^m
				\left(\e_{j_h} + \e_{i_1 j_h}\right) \in F $
				follows $ \a_{i_1} = (m - 1) \a_{j_1} $
			\State From $ \e_{i_1} + \sum_{h = 1}^m
				\left(\e_{j_h} + \e_{i_1 j_h}\right) \in F $
				follows $ \a_{i_1 j_1} = -\a_{j_1} $
			\State Let $ \hat{Y} \leftarrow Y \setminus \set{j_1, \ldots, j_m} $
				and $ \hat{I} \leftarrow \Is \setminus \set{I_1, I_2} $
			\State From $ \e_j \in F $ follows $ \a_j = 0 $ for all $ j \in \hat{Y} $
			\State From $ \e_i \in F $ follows $ \a_i = 0 $
				for all $ i \in I $ with $ I \in \hat{I} $
			\State From $ \e_{j_p} + \e_i + \e_{ij_p} \in F $
				follows $ \a_{ij_p} = 0 $ for all $ i \in I $ with $I \in \hat{I} $
				and $ p = 2, \ldots, m $
			\State From $ \e_j + \e_i + \e_{ij} \in F $ follows $ \a_{ij} = 0 $
				for all $ i \in I $ with $ I \in \hat{I} $ and $ j \in \hat{Y} $
			\State From $ \e_{j_1} + \e_{i_2} + \e_{i_2 j_1} + \e_i + \e_{ij_1} \in F $
				follows $ \a_{ij_1} = 0 $ for all $ i \in I $ with $ I \in \hat{I} $
			\State From $ \e_{i_1} + \sum_{r = 2}^m
				\left(\e_{j_r} + \e_{i_1 j_r} \right) + \e_j + \e_{i_h j} \in F $
				follows $ \a_{i_1 j} = 0 $ for all $ j \in \hat{Y} $
			\State From $ \e_{i_h} + \e_j + \e_{i_h j} \in F $
				follows $ \a_{i_hj} = 0 $ for all $ h = 2, \ldots, m $
				and $ j \in \hat{Y} $
		\end{algorithmic}
		\caption{Construction of affinely independent points on an \arrowone{} facet}
		\label{Alg:Arrow-1_Facets}
	\end{algorithm}
	
	Copying lifts further variables into the inequality.
	For each $ i \in I_1 \setminus \set{i_1} $, there are two possibilities:
	either all coefficients associated with $i$ (for the $x$- and $z$-variables) are~$0$,
	then similar steps as 12--18 need to be done,
	or they are identical to those of~$ i_1 $,
	then the construction can be adapted by replacing~$i$ with~$ i_1 $
	in \cref{Alg:Arrow-1_Facets}.
	For $ i \in I_2 \setminus \set{i_2, \ldots i_m} $,
	there are $m$~possibilities:
	again, either all coefficients are $0$,
	or they coincide with those corresponding
	to some node from $ k \in \set{i_2, \ldots, i_m} $,
	in which case we can replace~$i$ with~$k$ in \cref{Alg:Arrow-1_Facets}.
	Therefore, the \arrowone{} inequalities and all their copyings define facets.
\end{proof}

\subsubsection{\Arrowtwo{} inequalities}
The second class of facets can be stated as
\begin{equation}
	x_{i_1} + \sum_{p = 2}^m \y_{j_p} - \sum_{p = 1}^m z_{i_1 i_p}
		+ \sum_{p = 2}^m \left(z_{i_p j_1} - z_{i_p j_p}\right) \geq 0,
\label{equ:new_facet2}
\end{equation}
for distinct $ I_1, I_2 \in \Is $, $i_1 \in I_1 $,
pairwise distinct $ i_2, \ldots, i_m \in I_2 $,
pairwise distinct $ j_1, \ldots, j_m \in Y $
and some $ m \geq 3 $.
The support of these \emph{\arrowtwo{} inequalities}
is shown in \cref{fig:supp_arrow1} to the right.
Their validity follows from \cref{thm:1minequalities} by setting $ k = m - 1 $,
and, similar as before, they are facet-defining.
\begin{theorem}[\Arrowtwo{} facets]
	Inequality~\cref{equ:new_facet2} as well as all its copyings
	define facets of $ \zP(G, \Is) $
	if~$G$ is a complete bipartite graph.
	\label{arrow2-facets-2}
\end{theorem}
The proof is given in \cref{sec:arrow2-proof}.

%% file: separation.tex
\section{Separation algorithms}
\label{sec:sep}

For several of the facet classes derived in \cref{sec:fac-def-in},
we are able to give efficient separation algorithms.
Each of these separation routines
consists of an enumeration part over subsets of~$Y$,
where in each iteration an auxiliary combinatorial optimization problem
has to be solved.
We also discuss how these algorithms can be modified
in order to separate over all switchings and copyings as well
and what this means for their complexity.
For ease of exposition, we state all separation routines
for the case of a complete bipartite graph~$G$.

\subsection{Separation of cycle inequalities}

A well-known polynomial-time separation algorithm
for the cycle inequalities for $ QP $
is based on shortest-path-computations
and can be found in \cite{deza1997geometry}.
As the cycle inequalities for $ QP(G) $ remain valid for $ \zP(G, \Is) $,
this algorithm can still be used for separation;
however, not all cycle facets of $ QP(G) $
are still facets of $ \zP(G, \Is) $.
For a complete bipartite graph,
only cycles of size four define facets
and they dominate all longer cycle inequalities.
In \cref{alg:sep-cycle}, we show how to separate
over all cycle inequalities of size four and their copyings.
\begin{algorithm}	
	\begin{algorithmic}[1]
		\Require{$ (x, y, z) \in \R^{X \cup Y \cup E} $}
		\Ensure{Most violated cycle+copying inequality
			for each $ \j_1, \j_2 \in Y $ with $ \j_1 \neq \j_2 $.}
		\Function{Cycle+Copying-Separator}{$ (x, y, z) $}
			\For{each distinct $ \j_1, \j_2 \in Y$}
				\For{$ \I \in \Is $}
					\For{$ i \in \I $}
						\Lett{$ v^1_i $}{$ z_{i \j_1} - z_{i \j_2} $}
							{$ v^2_i $}{$ -x_{i} + z_{i \j_1} + z_{i \j_2} $}  
					\EndFor
					\Lett{$ S^1_{\I} $}{$ \set{i \in \I \colon v^1_i > 0} $}
						{$ S^2_{\I} $}{$ \set{i \in \I \colon v^2_i > 0} $} 
					\If{$ S^1_{\I} = \emptyset $} 
						\State $ p \leftarrow \argmax\set{v^1_i \mid i \in I} $
						\State $ S^1_{\I} \leftarrow \set{p} $
					\EndIf
					\If{$ S^2_{\I} = \emptyset $} 
						\State $ p \leftarrow
						\argmax\set{v^2_i \mid i \in I} $
						\State $ S^2_{\I} \leftarrow \set{p} $
					\EndIf			
					\Lett{$ \phi^1_{\I} $}{$ \sum_{i \in S^1_{\I}} v^1_i $}
						{$ \phi^2_{\I} $}{$ \sum_{i \in S^2_{\I}} v^2_i$ }
				\EndFor
				\State Select distinct $ \I_1, \I_2 \in \Is $
					which maximize $ \phi^1_{\I_1} + \phi^2_{\I_2} $
				\If{$ \phi^1_{\I_1} + \phi^2_{\I_2} - \y_{\j_1} > 0 $}
					\State \textbf{return} violated cycle inequality
						based on $ (\j_1, \j_2, I_1, I_2, S^1_{\I_1}, S^2_{\I_2}) $
				\EndIf
			\EndFor
		\EndFunction
	\end{algorithmic}
	\caption{Separation of cycle+copying inequalities}
	\label{alg:sep-cycle}
\end{algorithm}

A cycle+copying inequality is uniquely defined
by choosing distinct $ \j_1, \j_2 \in Y $
distinct $ I_1, I_2 \in \Is $,
an $ S^1 \subseteq I_1 $ and an $ S^2 \subseteq I_2 $.
Thus, there are exponentially many cycle inequalities with copying.
However, once $ j_1, j_2 $ are fixed (step~2),
one can calculate for each $ i \in I $ with $ \I \in \Is $ separately
the contribution to the left-hand side of the constraint
for the two cases $ i \in S^1 $ and $ i \in S^2 $ (steps~3--6).
Then one can select greedily the optimal subsets $ S^1_{I_1} $ and $ S^2_{I_1} $,
one for each $ \I \in \Is $ (steps~7--9)
and choose the two subsets which lead to the maximal combined violation (step~12).
In order to just separate over cycles with size four (without copying),
step~7 in the algorithm needs to be left out and we choose the then-branch in steps 8 and 9;
its complexity, however, does not decrease. 
The running time of the algorithm
is of the order $ \OO(\card{Y}^2 \card{X}) $,
including the greedy suboptimization.

\Cref{alg:sep-cycle} is written to separate over cycle inequalities
of the form~\cref{eq:cycle_general_class_short_copy},
but can easily be extended to separate over their switchings as well.
Note that such a switching is simply a lower bound instead of an upper bound
on the same left-hand-side expression.
By swapping all inequality and optimization senses in steps~7--13,
one can obtain the most violated switched cycle inequality simultaneously.

In \cref{sep:gen-lifted},
we give a template for the separation of $0$-lifted valid inequalities for $ BQP $
with bounded $Y$-support which includes \cref{alg:sep-cycle} as a special case.

\subsection{Separation of arrow inequalities}

In \cref{alg:sep-no1x}, we present a routine
to separate over all \arrowone{} inequalities.
It is again based on an enumeration part
and an auxiliary subproblem in each iteration,
which takes the form of a minimum-cost-circulation problem (MCCP) in this case.

For distinct $ \I_1, \I_2 \in \Is $ as well as some $ i_1 \in \I_1 $ and $ \j_1 \in Y$,
let $ J \coloneqq Y \setminus \set{j_1} $, let $ W \coloneqq I_2 \times J $,
and let $ H = (\tilde{V}, A) $ be a directed graph
with node set $ \tilde{V} \coloneqq \set{s} \cup \I_2 \cup J $.
The arc set~$A$ contains the arcs $ (s, i) $
for $ i \in \I_2 $, all $ (i, j) \in W $
and $ (j, s) $ for $ j \in J $.
Each arc has a capacity of at most one.
The graph~$H$ is shown in \cref{fig:arrow1-graph}.
\begin{figure}[H]
	\null\hfill\null
	\begin{subfigure}[c]{0.4\textwidth}
		\begin{tikzpicture}
			\node[shape=circle,draw=black,fill=lightmauve] (Y) at (-1.4,1) {$s$};
			\node[shape=circle,draw=black,fill=orange!30] (A) at (0,0) {$i^2_{\abs{I_2}}$};
			\node[shape=circle,draw=black,fill=orange!30] (B) at (0,1) {$\vdots$};
			\node[shape=circle,draw=black,fill=orange!30] (C) at (0,2) {$i^2_1$};
			\node[shape=circle,draw=black,fill=red!30] (D) at (1.8,0) {$j_{\abs{Y}}$};
			\node[shape=circle,draw=black,fill=red!30] (E) at (1.8,1) {$\vdots$};
			\node[shape=circle,draw=black,fill=red!30] (F) at (1.8,2) {$j_2$} ;
			\node[shape=circle,draw=black, fill=lightmauve] (Z) at (3.2,1) {$s$};
			
			\path [line] (C) -- node [midway,above=0.6em ] {$z_{i\j}+x_{i_1}$} (F);
			
			\path [line] (Y) -- node [near start, above=1em ] {$-z_{i\j_1}$} (C);
			
			\path [line] (F) -- node [near end, above=0.5em ] {$-z_{i_1\j}$} (Z);
			
			\path [line] (Y) -- (A);
			\path [line] (Y) -- (B);
			\path [line] (Y) -- (C);
			
			\path [line] (A) -- (D);
			\path [line] (A) -- (E);
			\path [line] (A) -- (F);
			
			\path [line] (B) -- (D);
			\path [line] (B) -- (E);
			\path [line] (B) -- (F);
			
			\path [line] (C) -- (D);
			\path [line] (C) -- (E);
			\path [line] (C) -- (F);
			
			\path [line] (E) -- (Z);
			\path [line] (D) -- (Z);
		\end{tikzpicture}
	\end{subfigure}
	\null\hfill\null
	\begin{subfigure}[c]{0.55\textwidth}
		\begin{equation*}
			c_a \coloneqq \begin{cases}
			-z_{i \j_1} & \text{if } a = (s, i) \text{ for } i \in \I_2\\
			z_{i \j} + x_{i_1} & \text{if } a = (i, j) \in W\\
			-z_{i_1 \j} & \text{if } a = (j, s) \text{ for } j \in J.
			\end{cases}
		\label{equ:c_arrow1}
		\end{equation*}
	\end{subfigure}
	\null\hfill\null
	\caption{Graph~$H$
		for $ I_2 = \set{i^2_1, \ldots, i^2_{\abs{I_2}}} $
		and $ J = \set{j_2, \ldots, j_{\card{Y}}} $ (left).
		The node~$s$ is duplicated in the figure
		for better readability.
		The unit costs for each arc (right).}
	\label{fig:arrow1-graph}
\end{figure}
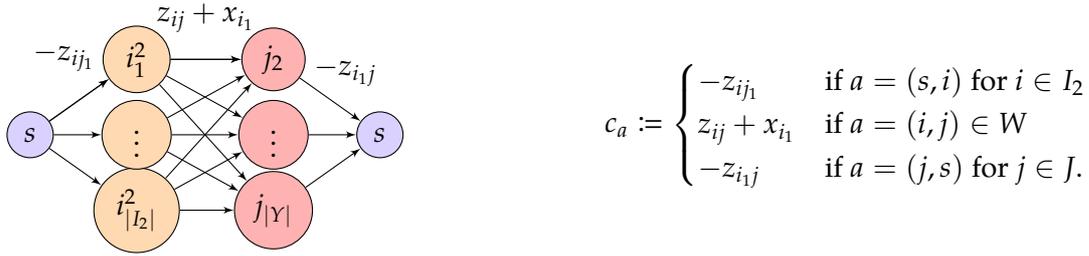
\begin{algorithm}
	\begin{algorithmic}[1]
		\Require{$ (x, y, z) \in \R^{X \cup Y \cup E} $}
		\Ensure{Most violated \arrowone{} inequality	
			for distinct $ \I_1, \I_2 \in \Is $,
			$ i_1 \in \I_1 $ and $ \j_1 \in Y $.}
		\Function{\arrowone-Separator}{$ x, y, z $}
		\For{each distinct $ \I_1, \I_2 \in \Is $,
			each $ i_1 \in \I_1 $ and each $ \j_1 \in Y $}
			\State Compute minimum-cost circulation on $H$.
			\Let{$o$}{optimal value determined in step 3}
			\Let{$ \tilde{A} $}{$ \set{(i, j) \in W \mid (i,j)$ is part of the optimal circulation$}$}
			\Let{$ \set{j_2, \ldots j_m} $}{$ \set{j \in J \mid (i, j) \in \tilde{A}} $} \Comment{naming the elements}
			\For{$n \in \set{2, \ldots, m}$}
				\State Let $ i_n \in \set{i \in I_2 \mid (i, j_n) \in \tilde{A}} $ \Comment{set contains only one element}
			\EndFor
			\If{$ o + y_{j_1} - z_{i_1 \j_1} < 0 $}
				\State \textbf{return} violated inequality
					based on $ (\I_1, \I_2, i_1, \j_1, j_2, \ldots j_m, i_2, \ldots, i_m) $
			\EndIf
		\EndFor
		\EndFunction
	\end{algorithmic}
	\caption{Separation of \arrowone{} inequalities}
	\label{alg:sep-no1x}
\end{algorithm}

An \arrowone{} inequality is uniquely defined
by choosing distinct $ \I_1, \I_2 \in \Is $,
an $ i_1 \in \I_1 $ and a $ \j_1 \in Y $ on the one hand
as well as $ i_2, \ldots, i_m \in I_2 $ and $j_2, \ldots, j_m \in J$ on the other hand.
The former possibilities are enumerated in step~2 of \cref{alg:sep-no1x},
while the latter is determined via an MCCP
in steps~5--12 of the algorithm.

The running time of the algorithm
is $ \OO(\card{\Is}^2 \card{Y} I^{\max} D(\I^{\max}, \card{Y})) $,
where $I^{\max} \coloneqq \max \set{\card{I} \mid I \in \Is}$ is the cardinality of the largest subset in $\Is$ and $ D(\I^{\max}, \card{Y}) $ is the running time of the algorithm
to solve the MCCP on~$H$.
Via the push-relabel algorithm, this is possible, for example,
in $ \OO(\card{\tilde{V}}^2 \sqrt{\card{A}}) $ (see \cite{cheriyan1989analysis}),
where $ \card{\tilde{V}} = \I^{\max} + \card{\Y} + 1 $
and $ \card{A} = \card{\I^{\max}} \card{\Y} + \card{\I^{\max}} + \card{\Y} $.
Observe that we did not have to specify~$m$ in \cref{alg:sep-no1x},
but that it is determined by the amount of the flow through node~$s$.

The \arrowtwo{} inequalities can be separated
via a slight modification of the algorithm
by changing the cost function for the MCCP
to reflect the support of the \arrowtwo{} inequalities.
Finally, how to separate over the switchings and copyings
of both types of arrow inequalities is described
\cref{sec:arrow1-switching} and \cref{sec:ip-separierer} respectively.

%% file: comp_results.tex
\section{Computational results}
\label{sec:comp-res}

In the following, we evaluate the computational benefit
of our polyhedral findings and separation routines for $ \zP(G, \Is) $.
We first present an empirical study
of the strength of the facets classes we have derived.
Then we show that the resulting cutting planes
lead to vastly improved performance
in solving a novel type of pooling problem
for instances based on real-world data.

For the computations, we used a server
with Intel Xeon E5-2690v2 3.00~GHz processors, 128 GB RAM, 1~core
and \emph{Gurobi~9.0.2} (\cite{gurobi}) as a quadratic solver.

\subsection{Strength of relaxation of the facet classes}
\label{sec:res-random}

We consider the optimization problem $ \max\set{c^T(x, y, z) \mid
	(x, y, z) \in P(G, \Is)} $
for some $ c \in \R^{X \cup Y \cup E} $.
The standard LP relaxation of this problem,
which state-of-art solvers like \emph{Gurobi} use, is given by 
\begin{equation}
	\max\SSet{c^Tx}{
		\begin{split}
			& z_{ij} \geq 0\quad \forall \set{i, j} \in E,
				\quad x_i + y_j - z_{ij} \leq 1\quad \forall \set{i, j} \in E,\\
			& z_{ij} \leq x_i\quad \forall i \in X,
				\quad z_{ij} \leq y_j\quad \forall j \in Y,\quad x \in X^{\Is}
		\end{split}
	}.
	\label{Eq:LP_Relaxation}
\end{equation}
Apart from the multiple-choice constraints,
this relaxation includes the famous \emph{McCormick}-inequalities,
which are a weaker form of the RLT inequalities we derived.
In this computational study, we investigate for several of the facet classes
presented in \cref{sec:fac-def-in} by how much they close the gap
between integer and LP optimum.
To this end, we generated~7 differently sized complete bipartite graphs~$G$
for which we solve the problem using 10~different random objective functions~$c$ each.
The coefficients of~$c$ are sampled uniformly and independently
from the interval $ [-10, 10] $.
As most of the derived facet classes are exponential in size,
we separated them using the routines from \cref{sec:sep}
until no more violated inequalities were found.
Only for the RLT inequalities, we directly added all of them from the start. 

In \cref{tb:random_rel_reduction}, we state the resulting gaps
averaged over all 10~instances with the same underlying graph.
The first column displays the short names of the different classes:
LP for the LP relaxation with no additional cuts,
\emph{RLT} for the RLT inequalities,
\emph{C}~for the cycle inequalities,
\emph{A1} for the \arrowone{} inequalities
and \emph{A2} for the \arrowtwo{} inequalities.
An additional \emph{S} at the end stands for all inequalities that are obtained
via the switching operation, and the same holds for~\emph{C} with the copy operation.
For the row \emph{All}, we first added all RLT inequalities
and then ran the separation routines for CC, A1S, A1C, A2S and A2C
from \cref{sec:sep} in an iterative loop,
inserting all violated inequalities of each class
before resolving the relaxation
until no more violated inequalities of any class were found.  
The top line in each column indicates the size of the instances.
These are stated in the format $ \card{\Is} $-$ \card{I} $-$ \card{Y} $
with all $ I \in \Is $ being of the same cardinality.
For example, \emph{$5$-$5$-$10$} means $5$ subsets in the partition
with $5$~nodes each and $10$~nodes in~$Y$.
The instances of type \emph{$10$-$*$-$25$} have a partition with $10$~subsets
of sizes $ 1, \ldots, 10 $ and $25$~nodes in~$Y$.
\begin{table}[h]
	\centering
	\caption{Average gaps in \% between integer and LP optimum
		for instances of different size and for various facet classes
		added to the linear relaxation}
	\label{tb:random_rel_reduction}
	\begin{tabular}{ l|rrr|rrr|r }
		\toprule
		& $5$-$5$-$10$ & $10$-$10$-$10$ & $15$-$15$-$10$
		& $5$-$5$-$20$ & $5$-$5$-$40$ & $5$-$5$-$60$ & $10$-$*$-$25$\\
		\midrule
		LP	& 17.83 & 21.81 & 22.47 & 29.03 & 37.85 & 42.28 & 40.66\\
		\midrule
		RLT	& 8.02 & 13.39 & 14.01 & 20.90 & 31.66 & 36.57 & 31.88\\
		C	& 3.63 & 11.90 & 14.30 & 17.05 & 30.51 & 36.08 & 28.82\\
		CC	& 0.00 & 0.00 & 0.00 & 0.00 & 1.99 & 4.43 & 4.05\\
		A1	& 1.95 & 7.55 & 8.97 & 10.65 & 24.71 & 30.34 & 21.18\\
		A1S	& 0.00 & 0.13 & 0.42 & 3.64 & 17.31 & 23.47 & 11.73\\
		A1C	& 0.28 & 1.85 & 2.82 & 6.41 & 18.41 & 23.58 & 15.50\\
		A2	& 9.52 & 16.85 & 18.14 & 22.82 & 34.12 & 38.87 & 33.51\\
		A2S	& 3.00 & 11.66 & 14.30 & 17.05 & 30.51 & 36.08 & 27.23\\
		A2C	& 0.31 & 2.74 & 4.13 & 7.40 & 19.78 & 24.88 & 17.20\\
		All	& 0.00 & 0.00 & 0.00 & 0.00 & 0.45 & 1.48 & 1.64\\
		\bottomrule
	\end{tabular}
\end{table}

We see that the CC inequalities have by far
the highest impact in reducing the gap.
On the first three instance classes,
they are the only group that is able to completely close the gap.
Furthermore, they are considerably stronger than the cycle inequalities without copying,
which is notable because there is practically no additional effort
in separating over the copyings as well.
We remark that on the larger instances,
we found about three times as many violated CC inequalities as C inequalities.
Note that C lies in the intersection of the classes A1S and A2S. 
While the A2S inequalities are empirically only marginally stronger
than the C inequalities,
the A1 inequalities perform much better than both other classes.
This performance can be improved significantly by passing over
to the larger set of the A1S inequalities.
For the smaller instances, they are almost as strong as the CC inequalities;
however, for large instances the latter are far superior.
The RLT inequalities contribute notably to closing the gap for small instances
and are still helpful for large instances,
especially because they achieve this effect
with relatively few cutting planes to be added.
When we look at the strength of all inequalities taken together,
we see that even for the largest instances
there only remains a very small integrality gap.
This means that the facet classes studied here
already describe the polytope under consideration very well.

Summarizing, we conclude that the CC inequalities
are the most important ones to separate,
and, favourably, also the computationally easiest ones to separate.
For larger instances, the other inequalities further help to reduce the gap.
Indeed, in \cref{sec:further-study}
we show that the CC and the A2C inequalities together
are the best combination of any two types of cutting planes in our study.
In \cref{sec:num-cutting-planes}, we also show in detail
how many cutting plans are required
to obtain the results from \cref{tb:random_rel_reduction}.

\subsection{Computational study on real-world pooling instances}
\label{comp:pooling}

We finally present a computational study
on real-world instances of a special variant of the pooling problem,
where we use some of the derived cuts
to improve the solution process.
First, we describe the underlying pooling model,
then we give some details about the origin of the data.
The results on our benchmark set will show
that exploiting the multiple-choice structure in the problem
yields a significant solution time benefit.

\subsubsection{Pooling model}

The pooling problem is a well-known continuous and non-convex optimization problem
(see \cite{gupte2017relaxations} for a detailed introduction). 
From a practical point of view,
it arises whenever different raw materials of varying quality
need to be mixed over multiple production stages
in order to obtain final products with given quality requirements.
In the mathematical sense, it is a combination of the two famous problems
minimum-cost flow and blending.

Typically, the pooling problem is represented
over a directed graph $ G = (V, A) $ 
with a node set $ V = I \cup P \cup O $
which is split into a set of inputs~$I$, pools~$P$ and outputs~$O$
as well as a set of quality specifications~$S$. 
We assume that there exist no arcs between pools,
which is the standard setting studied in the literature.
For each $ i \in I $ and $ s \in S $,
$ \lambda_{is} \in \R_+ $ denotes the value of specification~$s$ of input~$i$. 
Further, for each output $ o \in O $ and each specification $ s \in S $,
a minimum level $ \mu^{\min}_{os} \in \R_+ $
and a maximum level $ \mu^{\max}_{os} \in \R_+ $ are given.
The maximum available amount of raw material
at input~$ i \in I $ is given by $ b_i \in \R_+ $,
and for $ o \in O$, $ d_o \in \R_+ $ is the demand at output~$o$.
For each pool~$ l \in L $, $ I_l \subseteq I $ shall denote
the subset of inputs which have an arc pointing to~$l$.
The goal of pooling is then to send as much flow as possible
from inputs to outputs through the graph
while ensuring that the specification limits hold.

Up to this point, we have described 
the classical pooling problem from the literature.
In our application, there are additional flow restrictions
in the form of multiple-choice constraints,
which correspond to the \emph{recipes} used in tea production.
They specify percentagewise how much of each raw material
is needed to produce a given final product.
This requirement leads to the following definitions:
for each pool~$ l \in L $, let $ \Is_l \coloneqq \set{I^1_l, \ldots, I^{r_l}_l} $
be a partition of $ I_l $ with $ r_l $-many subsets,
where each subset contains only inputs of the same material.
Moreover, we are given a factor $ 0 \leq \sigma^h_l \leq 1 $
for each $ l \in L $ and $ h = 1, \ldots, r_l $
stating the desired proportion of total flow
from the associated inputs in $ \Is_l $ arriving at~$l$.
These factors satisfy $ \sum_{h \in [r_l]} \sigma^h_l = 1 $ for all $ l \in L $.
We now introduce the two model formulations which
we will compare in our study.

\paragraph{$q$-Formulation for the pooling problem}

A well-known formulation for the pooling problem
is the so-called \emph{$q$-formulation}
(see \cite{gupte2017relaxations}),
which can easily be extended
for the additional multiple-choice constraints.
Let variable $ y_{ij} \in \R_+ $ be the flow on arc $ a = (i, j) \in A $,
and let variable $ q_{il} \in \R_+ $ be the proportion of flow
from input $ i \in I_l $ to pool $ l \in L $.
For notational simplicity, we will use the flow variables~$ y_{ij} $
with the understanding that $ y_{ij} \equiv 0 $ if $ (i, j) \notin A $.
Furthermore, we define the auxiliary variables $ v_{ilo} \in \R_+ $
for $ i \in I_l $, $ l \in L $ and $ o \in O $
for the product of $ y_{ij} $ and $ q_{il} $.
The pooling problem can then be modelled
via the following constraints:
\begin{alignat}{1}
	y_{il} & \leq b_i \quad \forall i \in I, \forall l \in L,
		\quad y_{lo} \leq d_o \quad \forall o \in O, \forall l \in L,\label{pool:1}\\
	\sum_{i \in I} y_{il} &= \sum_{o \in O} y_{lo}
		\quad \forall l \in L,\label{pool:2}\\
	\sum_{i \in I^h_l} y_{il} &= \sigma_h \sum_{o \in O} y_{lo}
		\quad \forall l \in L, \forall h = 1, \ldots, r_l,\label{pool:3}\\
	\sum_{i \in I_l} q_{il} &= 1 \quad \forall l \in L,\label{pool:4}\\
	v_{ilo} &= q_{il} y_{lo}
		\quad \forall l \in L, \forall i \in I_l, \forall o \in O,\label{pool:5}\\
	y_{il} &= \sum_{o \in O} v_{ilo}
		\quad \forall l \in L, \forall i \in I_l, \label{pool:6}\\
	\mu^{\min}_{os} \sum_{i \in I \cup L} y_{io}
		& \leq \sum_{i \in I} \lambda_{is} y_{io}
			+ \sum_{l \in L} \sum_{i \in I_l} \lambda_{is} v_{ilo}
			\leq \mu^{\max}_{os} \sum_{i \in I \cup L} y_{io}
			\quad \forall o \in O, \forall s \in S.\label{pool:7}
\end{alignat}
Inequalities~\cref{pool:1} specify the upper bounds on the flow variables,
and equations~\cref{pool:2} model the flow conservation.
The novel recipe structure is modelled via~\cref{pool:3}.
Inequalities~\cref{pool:4} require
that the proportion variables at each pool sum up to one.
Constraint~\cref{pool:5} defines the auxiliary variables
to linearize the bilinear terms
while constraint~\cref{pool:6} ensures consistency
between the $q$ and $y$-variables.
Finally, \cref{pool:7} demands
that the specification bounds at the outputs be respected.

The aim is to maximize the total flow arriving at the outputs,
which is realized via the objective function $ \max \sum_{l \in L} \sum_{o \in O} y_{lo} $.

\paragraph{$q$+cuts formulation for the pooling problem}

The equations given by
\begin{equation}
	\sum_{i \in I^h_l} q_{il} = \sigma_h \quad \forall l \in L, \forall h = 1, \ldots, r_l
	\label{equ:qrecipe1}
\end{equation}
are valid for the pooling problem~\cref{pool:1}--\cref{pool:7}, 
which can easily be checked
by rearrangement of \cref{pool:3}, \cref{pool:5} and \cref{pool:6}.
Observe that for each pool $ l \in L $
as well as the corresponding subvectors $ q_l $, $ y_l $ and $ v_l $
of $q$, $y$ and $v$, the set
\begin{equation}
	\set{(q_l, y_l, v_l) \in \R_+^{I_l \times O \times (I_l \times O)} \mid
		(q_l, y_l, v_l) \text{ fulfils } \cref{pool:1}, \cref{pool:5}
			\text{ and } \cref{equ:qrecipe1}}
	\label{Eq:qyv-P}
\end{equation}
is a scaled and low-dimensional version
of the boolean quadric polytope with multiple-choice constraints~$P$
on a complete bipartite graph.
The low-dimensionality is due to~\cref{equ:qrecipe1}
being an equation instead of an inequality.

In \cref{sec:traf_low_dim},
we derive the following valid equations for \cref{Eq:qyv-P},
which we call the RLT equations:
\begin{equation}
	\sum_{i \in I^h_l} v_{ilo} = \sigma_h y_{lo}
		\quad \forall o \in O, l \in L, \forall h = 1, \ldots, r_l.
\label{equ:rlt-low-dim2}
\end{equation}
There we also show that valid constraints can easily be converted
between the full-dimensional case and the low-dimensional case.

Note that for each pool $ l \in L $, the set
\begin{equation}
	\set{(q_l, y_l, v_l) \in \R_+^{I_l \times O \times (I_l \times O)} \mid
		(q_l, y_l, v_l) \text{ fulfils } \cref{pool:1}, \cref{pool:4}
			\text{ and } \cref{pool:5}}
	\label{Eq:qyv-pq}
\end{equation}
also defines a low-dimensional instance of~$P$.
In this special case, where there is only a single multiple-choice constraint,
the complete description for \cref{Eq:qyv-pq}
is given by the corresponding RLT inequalities
(see \cite[Theorem~1]{gupteworking} as well as \cref{cor:compl_onesubset}).
Adding these RLT inequalities to \cref{pool:1}--\cref{pool:7}
leads to the well-known \emph{$ pq $-formulation} for the pooling problem
(see \cite{gupte2017relaxations}).

In our second formulation for the pooling problem with recipes,
which we call \emph{$q$+cuts},
we add to the $q$-formulation \cref{pool:1}--\cref{pool:7}
the equations~\cref{equ:qrecipe1},
the RLT equations~\cref{equ:rlt-low-dim2}
and the RLT inequalities for \cref{Eq:qyv-P} for each pool.
The $q$+cuts formulation can be seen as a $ pq $-formulation for pooling with recipes,
in the sense described above.
We will demonstrate that it is vastly superior
on our real-world instances when compared to the pure $q$-formulation.

\subsubsection{Computational results}

Based on real-world data from a tea producing company,
we have created six instances of different sizes for
the pooling problem with recipes introduced above.
The amount of available raw material,
the quality specifications at the inputs
and typical levels for the upper and lower specification bounds at the outputs
as well as the recipes were provided by our industry partner.
Under this setting, we have simulated various possible customer request scenarios,
for which we took the following assumptions.
Each pool produces one final product,
which is then sent to $5$~different outputs (customers)
with different specification limits and demands.
From the provided reference quality levels,
we derived $5$ different sets of quality specifications
for the different outputs by randomly decreasing the given lower bound
and increasing the upper bound by up to 15\% each.
The demands were drawn uniformly and independently from the interval $ [200, 4000] $,
which yields typical demand values for this application.
On average, the inputs had $6$ measured quality specifications each
while the outputs had an average of $190$ specifications each.

In our experiments, we compared the performance
of the $q$-formulation versus the $q$+cuts formulation
within a time limit of 1~minute.
The results are found in \cref{tb:comp_study}.
\begin{table}[t]
	\null\hfill\null
	\caption{Performance comparison for the $q$- and $q$+cuts formulation
		over 6 instances of different size:
		either solution time in seconds or optimality gap in percent
		after after a time limit of 1~minute (left).
		Number of rows, columns, bilinear terms (BLT) and non-zeros (NZ)
		of instance~6 after presolve in multiples of one thousand (right).}
	\label{tb:comp_study}
	\vspace{1.25\baselineskip}
	\begin{minipage}[b]{0.65\textwidth}
		\centering
		\begin{tabular}{ rrrr|rr }
			\toprule
			Instance & $ \card{I} $ & $ \card{P} $ & $ \card{O} $ & $q$ &  $q$+cuts\\
			\midrule
			1 & 187 & 34 & 170 & 18.59\% & 0.42 s\\
			2 & 229 & 43 & 215 & 3.91 s & 0.71 s\\
			3 & 306 & 82 & 410 & 38.69\% & 3.40 s\\
			4 & 329 & 117 & 585 & 115.09\% & 0.14\%\\
			5 & 360 & 160 & 800 & 69.68\% & 0.40\%\\
			6 & 465 & 230 & \numprint{1150} & 61.42\% & 0.36\%\\
			\bottomrule
		\end{tabular}
    \end{minipage}
    \null\hfill\null
    \begin{minipage}[b]{0.33\textwidth}
    	\centering
		\begin{tabular}{ lrr }
			\toprule
					& $q$ & $q$+cuts\\
			\midrule
			\#Rows	& 57 & 59\\
			\#Cols	& 19 & 17\\
			\#BLT	& 12 & 11\\			
			\#NZ	& 157 & 164\\
			\bottomrule
		\end{tabular}
	\end{minipage}    
    \null\hfill\null
\end{table}
We see that the $q$-formulation
could solve only one out of the six instances
within the time limit, while the other instances
still have huge optimality gaps.
In contrast, the $q$+cuts formulation solves the first three instances
to optimality within seconds.
The gaps of the other three instances are almost closed after 1~minute.
To provide some more context, we have included the sizes
of the two formulations for the largest instance~6 after Gurobi's presolve.
Both formulations could be reduced to similar sizes,
where, however, the $q$+cuts formulation has 10\% less bilinear terms,
which is an important factor influencing solution time.
Furthermore, we observed that $q$+cuts found good primal solutions much faster;
we assume that Gurobi's heuristics were able to benefit
from the tighter relaxation it provides.
Altogether, this shows that the $q$+cuts formulation
is vastly superior to the $q$-formulation on our test set.

%% file: conclusion.tex
\section{Conclusions}
\label{sec:conclusion}

We have seen that the joint consideration
of separable bilinear terms and multiple-choice constraints
leads to a very rich combinatorial structure,
whose exploitation is also beneficial from a computational point of view.
Many symmetries of the bipartite boolean quadric polytope remain intact,
which holds for the $0$-lifted version of many facet-defining inequalities as well.
Several subcases even allow for a characterization
of the complete convex hull via reformulation-linearization inequalities.
At the same time, there are very interesting new symmetries and facet classes
which arise specifically due to the additional multiple-choice structure.
Notably, the switching operation and the novel copy operation
provide a lifting framework which is able to produce
a vast amount of facets out of basic facet classes.
Moreover, we gave separation routines for these facet classes,
most of which run in polynomial time under the assumption of bounded support.
All of these procedures allow for efficient implementation,
and the corresponding cutting planes
lead to an almost complete closure of the integrality gap
on instances with up to $85$~nodes in our experiments.
Finally, we demonstrated that the bipartite boolean quadric polytope
with multiple-choice constraints is an adequate model
for pooling problems with fixed input proportions (\ie recipes) at the pools.
Putting our insights into practice
allowed us to solve very hard real-world instances
to (near-)optimality within a minute,
to which a standard solver did not even come close in most cases.

%% file: appendix.tex
\section*{Appendix}

Here we give some of the details
about the results omitted above.
In \cref{sec:arrow2-proof}, we prove \cref{arrow2-facets-2},
which states that the \arrowtwo{} inequalities define facets of $ \zP(G, \Is) $.
The separation algorithms for switchings and copyings
of the arrow-1 and arrow-2 inequalities are given
in \cref{sec:arrow1-switching} and \cref{sec:ip-separierer} respectively.
A separation template for general $0$-lifted valid inequalities
from BQP is presented in \cref{sep:gen-lifted}.
In \cref{sec:traf_low_dim}, we describe how to obtain valid inequalities
for the lower-dimensional version of $ \zP(G, \Is) $
where the multiple-choice constraints have to be fulfilled with equality.
To complement the computational experiments presented in \cref{sec:res-random},
we first investigate the number of cutting planes found per facet class
in \cref{sec:num-cutting-planes}.
Finally, in \cref{sec:further-study}
we examine which class of facet-defining inequalities is the second-strongest
after the cycle+copying inequalities empirically.

\section{Proof of \cref{arrow2-facets-2}}
\label{sec:arrow2-proof}

\begin{proof}
	We fix some distinct $ I_1, I_2 \in \Is $, a $ i_1 \in I_1 $,
	pairwise distinct $ i_2, \ldots, i_m \in I_2 $,
	pairwise distinct $ j_1, \ldots, j_m \in Y $
	and some $ m \geq 3 $.
	Let $ a^T(x, y, z) \leq b $ be a facet-defining inequality
	which contains the face~$F$ induced by \cref{equ:new_facet2}.
	In~\cref{Alg:Arrow-2_Facets}, we indicate how to construct
	$ \card{X} + \card{Y} + \card{E}$ affinely independent points on $F$	
	to show that $a$ and~$b$ are multiples of the coefficients of \cref{equ:new_facet2}.	
	\begin{algorithm}
		\begin{algorithmic}[1]
			\State From $ 0 \in F $ follows $ b = 0 $
			\State From $ \e_{j_1} \in F $ follows $ \a_1 = 0 $
			\State From $ \e_{i_1} + \e_{j_1} + \e_{i_1 j_1} \in F $
				follows $ \a_{i_1 j_1} = -\a_{i_1} $
			\State From $ \e_{i_p} \in F $ follows $ \a_{i_p} = 0 $
				for all $ p = 2 ,\ldots, m $
			\For{$ p = 2, \ldots, m $}
				\State From $ \e_{i_p} + \e_{j_p} + \e_{i_pj_p} \in F $
					follows $ \a_{i_p j_p} = -\a_{j_p} $
				\State From $ \e_{i_1} + \e_{j_1} + \e_{j_p}
						+ \e_{i_1j_1} + \e_{i_1j_p} \in F $
					follows $ \a_{i_1 j_p} = -\a_{j_p} $
				\State From $ \e_{i_1} + \e_{i_p}
						+ \e_{j_1} + \e_{j_p}
						+ \e_{i_1j_1} + \e_{i_1j_p} + \e_{i_pj_1} + \e_{i_pj_p} \in F $
					follows that $ \a_{i_pj_1} = \a_{j_p} $
				\State From $ \e_{i_1} + \e_{i_p}
						+ \e_{j_p} + \e_{i_1 j_p} + \e_{i_p j_p} \in F $
					follows that $ \a_{i_1} = \a_{j_p} $
				\State From $ \e_{i_p} + \e_{i_q}
						+ \e_{j_p} + \e_{i_p j_p} + \e_{i_q j_p} \in F $
					follows $ \a_{i_q j_p} = 0 $
					for all $ q \in \set{2, \ldots, m} \setminus p $
			\EndFor
			\State $ \hat{Y} \leftarrow Y \setminus \set{j_1, \ldots, j_m} $
				and $ \hat{I} \leftarrow \Is \setminus \set{I_1, I_2} $
			\State From $ \e_j \in F $ follows $ \a_j = 0 $ for all $ j \in \hat{Y} $
			\State From $ \e_i \in F $ follows $ \a_i = 0 $
				for all $ i \in I $ with $ I \in \hat{I} $
			\State From $ \e_{j_p} + \e_{i_p} + \e_i + \e_{ij_p} + \e_{i_pj_p} \in F $
				follows $ \a_{i j_p} = 0 $ for all $ i \in I $ with $ I \in \hat{I} $
				and $ p = 2,\ldots, m $
			\State From $ \e_j + \e_i + \e_{ij} \in F $
				follows $ \a_{i j_q} = 0 $ for all $ i \in I $ with $ I \in \hat{I} $
				and $ j \in \hat{Y} $
			\State From $ \e_{j_1} + \e_i + \e_{i j_1} \in F $
				follows $ \a_{ij_1} = 0 $ for all $ i \in I $ with $ I \in \hat{I} $
			\State From $ \e_{i_1} + \e_{j_1} + \e_{i_1j_1} + \e_j + \e_{i_h j} \in F $
				follows $ \a_{i_1 j} = 0 $ for all $ j \in \hat{Y} $
			\State From $ \e_{i_h} + \e_j + \e_{i_h j} \in F $
				follows $ \a_{i_h j} = 0 $ for all $ h = 2, \ldots, m $
				and $ j \in \hat{Y} $
		\end{algorithmic}
		\caption{Construction of affinely independent points on an \arrowtwo{} facet}
		\label{Alg:Arrow-2_Facets}
	\end{algorithm}
	
	Copying lifts further variables into the inequality.
	For each $ i \in I_1 \setminus \set{i_1} $, there are two possibilities:
	either all coefficients associated with $i$ (for the $x$- and $z$-variables) are~$0$,
	then similar steps as 12--19 need to be done,
	or they are identical to those of $ i_1 $,
	then the construction can be adapted by replacing~$i$ with~$ i_1 $
	in \cref{Alg:Arrow-2_Facets}.
	For $ i \in I_2 \setminus \set{i_2, \ldots i_m} $,
	there are $m$~possibilities:
	again, either all coefficients are $0$,
	or they coincide with those corresponding
	to some node from $ k \in \set{i_2, \ldots, i_m} $,
	in which case we can replace~$i$ with~$k$ in \cref{Alg:Arrow-2_Facets}.
	Therefore, the \arrowtwo{} inequalities and all their copyings define facets.
\end{proof}

\section{Separation of \arrowone{}+switching inequalities}
\label{sec:arrow1-switching}

In order to separate over all switchings of an \arrowone{} inequality,
\cref{alg:sep-no1x} for the separation of the original \arrowone{} inequalities
needs some minor changes.
Firstly, the graph~$H$ from \cref{fig:arrow1-graph}
needs to be modified as shown in \cref{fig:switching}.

The two possible cycles through $ j' $ and $ j'' $
represent the decision whether to switch on~$ \set{j} $ for $ j \in J $ or not.
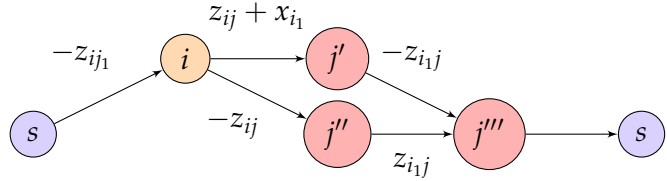
\begin{figure}[h]
	\centering
	\begin{tikzpicture}
		\node[shape=circle,draw=black,fill=lightmauve] (Y) at (-2,1) {$s$};
		\node[shape=circle,draw=black,fill=orange!30] (C) at (0,2) {$i$};
		\node[shape=circle,draw=black,fill=red!30] (E) at (2,1) {$j''$};
		\node[shape=circle,draw=black,fill=red!30] (F) at (2,2) {$j'$} ;
		\node[shape=circle,draw=black, fill=red!30] (Z) at (4,1) {$j'''$};
		\node[shape=circle,draw=black, fill=lightmauve] (R) at (6,1) {$s$};
	
		\path [line] (C) -- node [midway,above=0.6em ] {$z_{i\j}+x_{i_1}$} (F);
		\path [line] (Y) -- node [near start, above=1em ] {$-z_{i\j_1}$} (C);
		\path [line] (F) -- node [midway, above=0.5em ] {$-z_{i_1\j}$} (Z);
	
		\path [line] (C) -- node [near start,below=0.6em ] {$-z_{i\j}$} (E);
		\path [line] (E) -- node [midway,below=0.3em ] {$z_{i_1\j}$} (Z);
		\path [line] (Z) -- (R);
	\end{tikzpicture}
	\caption{The necessary changes to graph~$H$ to separate over all switchings
		of the \arrowone{} inequalities.
		Each node $ j \in J $ is split
		into three nodes $ j' $, $ j'' $ and $ j''' $.
		Then the arcs $ (i, j') $ and $ (i, j'') $ for $ i \in I_2 $
		as well as $ (j', j''') $, $ (j'', j''') $ and $ (j''', s) $ are introduced
		with costs as shown above.}
	\label{fig:switching}
\end{figure}
The number of nodes and edges in the modified graph
is $ \card{\tilde{V}} = \I^{\max} + 3 \card{\Y} + 1 $
and $ \card{A} = 4 \I^{\max} \card{\Y} + \I^{\max}+ \card{\Y} $ respectively,
so the solution time of the resulting MCCP
is moderately larger than before.

Secondly, the enumeration part in \cref{alg:sep-no1x} needs to be extended, too.
We loop twice over all $ j_1 \in \Y $,
once with arc costs directly as in \cref{fig:switching}
and once more an arc cost on $ (s, i) $ incorporating the switching on $ \set{j_1} $.
This basically doubles the running time of the overall algorithm. 

Furthermore the \arrowtwo{}+switching inequalities can be separated
via a slight modification of the algorithm
by changing the cost function for the MCCP
to reflect the support of the \arrowtwo{}+switching inequalities.

\section{Separation of \arrowone{}+copying inequalities}
\label{sec:ip-separierer}

Now we consider se\-pa\-ra\-ting over all copyings
of an \arrowone{} inequality,
which are given by
\begin{equation}
	\sum_{i \in S_1}\left( (m - 1) x_{i} - \sum_{p = 1}^m z_{i j_p} \right) + \y_{j_1} 
		+ \sum_{p = 2}^m \sum_{i \in S_p}\left(-z_{i j_1} + z_{i j_m} \right) \geq 0,
	\label{equ:new_facet1-copy}
\end{equation}
for distinct $ I_1, I_2 \in \Is $, a non-empty subset $ S_1 \subseteq I_1 $,
pairwise disjoint and non-empty subsets $ S_2, \ldots, S_m \subseteq I_2 $
and pairwise distinct $ j_1, \ldots, j_m \in Y $
with $ m \geq 3 $.
To this end, we replace \cref{alg:sep-no1x} by \cref{alg:sep-arrow-copy}.
The inner optimization subproblem
is now an integer MCCP on the same graph,
as described in the following.

For distinct $ \I_1, \I_2 \in \Is $ as well as $ \j_1 \in Y$,
let $ J \coloneqq Y \setminus \set{j_1} $,
and let $ H = (\tilde{V}, A) $ be a directed graph
with node set $ \tilde{V} \coloneqq \set{s} \cup \I_2 \cup J $.
The arc set~$A$ contains the arcs $ (s, i) $
for $ i \in \I_2 $, all $ (i, j) \in I_2 \times J $
and $ (j, s) $ for $ j \in J $.

Let  $ u_{ij} \in \N $ for $ (i, j) \in A $ be the flow variables.
Variables $ b_j \in \F $ for $ j \in J $
stand for the decision to use arc $ (j, s) \in A $.
Further, variables $ r_i \in \F $ decide if node~$ i \in I_1 $
is chosen to be in set~$ S_1 $.
Finally, variables $ p_{ij} \in \F $ for $ i \in I_1 $ and $ j \in J $
model the product between $ b_j $ and $ r_i $.
The constraints are then given by
\begin{subequations}
\begin{alignat}{1}
u_{ij} &\leq 1, \quad (i,j) \in \tilde{E} \setminus \set{(j,s) : j \in J} \label{arrow+copy:1}\\
\sum_{j \in J} b_j &\geq 1, \quad \label{arrow+copy:2}\\
\sum_{i \in I_1} r_i &\geq 1, \quad \label{arrow+copy:3}\\
u_{si} &= \sum_{j \in J} y_{ij}, \quad \forall i \in I_2, \label{arrow+copy:4}\\
u_{js} &= \sum_{i \in I_2} y_{ij}, \quad \forall j \in J, \label{arrow+copy:5}\\
u_{js} &\geq b_j, \quad \forall j \in J, \label{arrow+copy:6}\\
u_{js} &\leq \abs{I_2}b_j, \quad \forall j \in J, \label{arrow+copy:7}\\
p_{ij} &\leq r_i, \quad \forall i \in I_1, j \in J, \label{arrow+copy:8}\\
p_{ij} &\leq b_j, \quad \forall i \in I_1, j \in J, \label{arrow+copy:9}\\
r_i + b_j - p_{ij} &\leq 1, \quad \forall i \in I_1, j \in J. \label{arrow+copy:10}
\end{alignat}
\label{equ:feas_set}
\end{subequations}
Constraints \cref{arrow+copy:1} enforce the arc capacities,
\cref{arrow+copy:2} ensure that at least one $ j \in J $ is chosen,
\cref{arrow+copy:3} require to choose at least one $ i \in I_1 $,
\cref{arrow+copy:4,arrow+copy:5} are the flow conservation constraints,
\cref{arrow+copy:6} demand that there be flow on $ (j, s) $
if $ j \in J $ is chosen,
\cref{arrow+copy:7} ensure that there are at most $ \card{I_2} $ units of flow
if $ j \in J $ is chosen and that there is no flow otherwise,
\cref{arrow+copy:8,arrow+copy:9,arrow+copy:10} model $ p_{ij} = r_i b_j $.

The objective can be stated as
\begin{equation}
	\min \sum_{i \in I_2} z_{ij_1} u_{si}
		+ \sum_{i \in I_2, j \in J} z_{ij} u_{ij}
		+ \sum_{i \in I_1, j \in J} (-z_{ij} + x_i) p_{ij}
		+ \sum_{i \in I_1} (-z_{ij_1} r_i).
	\label{equ:obj}
\end{equation}

\begin{algorithm}[h]
	\begin{algorithmic}[1]
		\Require{$ (x, y, z) \in R^{X \cup Y \cup E} $}
		\Ensure{Most violated \arrowone{}+copy inequality
			for each $ \I_1, \I_2 \in \Is, \I_1 \neq \I_2, \j_1 \in Y $.}
		\Function{\Arrowone{}+Copy-Separator}{$ x, y, z $}
			\For{$ \I_1, \I_2 \in \Is, \I_1 \neq \I_2,\j_1 \in Y $}
			\State Solve integer problem defined by \cref{equ:feas_set} and \cref{equ:obj}
			\Let{$o$}{objective value of the integer problem in step 3}
			\Let{$ \bar{u}, \bar{b}, \bar{r}, \bar{p} $}
				{optimal solution of the integer MCCP in step 3}
			\If{$ o + y_{j_1} < 0 $}
				\State $ S_1 \leftarrow \set{i \in I_1 \mid \bar{r}_i = 1} $
 				\Let{$ \set{j_2, \ldots, j_m} $}{$ \set{j \in J \mid \bar{b}_j = 1} $}
 					\Comment{naming the elements}
				\For{$n \in \set{2, \ldots, m}$}
					\State $ S_n \leftarrow
						\set{i \in I_2 \mid (i, j_n) \in W, \bar{u}_{ij_n} = 1} $
				\EndFor
				\State \textbf{return} violated inequality
					based on $ (\I_1, \I_2, S_1, S_2, \ldots S_m, \j_1, \j_2, \ldots \j_m) $
			\EndIf
			\EndFor
		\EndFunction
	\end{algorithmic}
	\caption{Separation algorithm for \arrowone{}+copying inequalities}
	\label{alg:sep-arrow-copy}
\end{algorithm}

An \arrowone{}+copying inequality is uniquely defined
if distinct $ \I_1, \I_2 \in \Is $, a non-empty subset $ S_1 \subseteq \I_1 $,
a $ \j_1 \in Y $, pairwise distinct $ j_2, \ldots j_m \in J $
and non-empty, pairwise disjoint $ S_2, \ldots, S_m \subseteq \I_2 $ are chosen.
The possible choices for $ \I_1 $, $ \I_2 $ and $ \j_1 $ are again enumerated.
In step~3 we solve the integer MCCP. 
It determines adaptively the optimal value for~$m$.
The optimal~$ S_1 $ is then extracted from the $r$~variables in step~7,
and the optimal $ j_2, \ldots, j_m $ are derived
from the values of the $b$-variables in step~8. 
Finally, the optimal $ S_2, \ldots S_m $ are calculated in steps 9--11,
based on the optimal flow in the graph.

The \arrowtwo{}+copying inequalities can be separated
via a slight modification of the integer program
by changing the cost function for the minimum-cost-circulation problem
to reflect the support of the \arrowtwo{}+copying inequalities.

\section{Separation of general lifted inequalities from BQP}
\label{sep:gen-lifted}

For a subset uniform graph~$G$, any valid inequality for the $ BQP $
on the corresponding dependency graph~$ \GG $ can be $0$-lifted to $ \zP(G, \Is) $.
If $ \card{\supp_Y} = m $ for some $ m \geq 1 $,
such an inequality is of the form
\begin{equation*}
	\sum_{k = 1}^h a_{i_k} x_{i_k} + \sum_{p = 1}^m a_{j_p} y_{j_p}
		+ \sum_{k = 1}^h + \sum_{p = 1}^m (a_{i_k j_p} z_{i_k j_p}) \leq \delta,
\end{equation*}
with $ i_1 \in I_1, \ldots, i_h \in \I_h $
for pairwise distinct $ I_1, \ldots, I_h \in \Is $
and pairwise distinct $ j_1, \ldots, j_m \in Y $.
We now assume that for a specific class of constraints to separate over,
the coefficient vector~$a$ is already uniquely determined
by the above choice of $ i_1, \ldots, i_h $ and $ \j_1, \ldots, \j_m $.
The cycle or the $ I_{mm22} $ Bell inequalities would be examples here.
For such $ BQP $ inequalities with bounded $Y$-support,
we can give a separation template
which generalizes \cref{alg:sep-cycle}
for the cycle inequalities.
It is shown in \cref{alg:sep-lift-sep}.
\begin{algorithm}
	\begin{algorithmic}[1]
		\Require{$ (x, y, z) \in \R^{X \cup Y \cup E},
			m \in [\card{Y}], h \in [\card{\Is}]$}
		\Ensure{Most violated inequality for pairwise distinct $ \j_1, \ldots, \j_m \in Y $}
		\Function{Lift-Separator}{$ (x, y, z) $}
			\For{pairwise distinct $ \j_1, \ldots, \j_m \in Y $}
				\For{$ \I \in \Is, k \in \set{1, \ldots, h} $}
					\Let{$ c^I_k $}{$ \max\set{a_{i_k} x_r
						+ \sum_{p = 1}^m a_{i_k j_p} z_{r j_p} \mid
							r \in I} $}
				\EndFor
				\State Solve the $h$-cardinality assignment problem with objective $c$
				\Let{$o$}{objective value of the assignment problem in step 6}
				\Let{$v$}{$ \sum_{s = 1}^m a_{j_s} \y_{j_s} $}
				\If{$ o + v > \delta $}
					\Let{$ \set{(I_1, 1), \ldots (I_h, h)} $}{optimal solution in step 6}
					\For{$k \in \set{1, \ldots, h}$}
						\Let{$ i_k $}{$ \argmax\set{a_{i_k} x_r
							+ \sum_{p = 1}^m a_{i_k j_p} z_{r j_p} \mid r \in I_k} $}
					\EndFor
				\State \textbf{return} violated inequality
					based on $ (\j_1, \ldots, \j_m, I_1, \ldots, I_h, i_1, \ldots, i_h) $
				\EndIf
			\EndFor
		\EndFunction
	\end{algorithmic}
	\caption{Separation algorithm for a given class of $0$-lifted facets from $ BQP $}
	\label{alg:sep-lift-sep}
\end{algorithm}

A $0$-lifted $ BQP $ inequality of the considered class
is uniquely defined if we chose pairwise distinct $ \j_1, \ldots, \j_m \in Y $,
pairwise distinct $ I_1, \ldots, I_h \in \Is $
and one node from each $ I_1, \ldots, I_h $.
In general, there are exponentially many possible inequalities of this type.
For each fixed choice of nodes $ \j_1, \ldots, \j_m \in Y $ (step~2),
we define the values~$ c^I_k $ for $ \I \in \Is $ and $ k = 1, \ldots, h $
as the largest contribution that any element $ i \in \I $
would add to the left-hand side of the inequality to be separated (steps~3--5)
if~$i$ was assigned to the coefficients associated with~$k$.
Then an $h$-cardinality maximum assignment problem (see \cite{pentico2007assignment})
with objective~$c$ needs to be solved (step~6).
It determines which $ I \in \Is $ is assigned to which $ k \in \set{1, \ldots, h} $.
From the optimal assignment, the best possible choice
of the subsets $ I_1, \ldots I_h $ can be extracted (step~10),
from which we can derive the optimal $ i_1, \ldots, i_h $ (steps~11--13).

If the considered inequality class from $ BQP $
contains inequalities with $y$-supports of different sizes,
the algorithm needs to be run for every possible~$m$.
In order to separate also over all copies of these inequalities,
the maximum in step~4 needs to be taken
only over the sum of all elements that would make a positive contribution
to the left-hand side, similar to steps~7--9 in \cref{alg:sep-cycle}.

In fact, \cref{alg:sep-cycle} (extended to include the copyings)
is a special case of \cref{alg:sep-cycle},
where we have $ m = 2 $, and where solving the assignment subproblem
can be done via a greedy algorithm.
The running time of \cref{alg:sep-cycle}
is $ \OO((\card{Y})!/m! D(\card{\Is}, h)) $,
where $ D(\card{\Is}, h) $ is the running time
needed to solve the $h$-cardinality assignment subproblem.
It can be solved as a minimum-cost flow problem on a graph,
where the node set has size $ \card{V} = \card{\Is} + h + 2 $
and the edge set has size $ \card{E} = \card{V} \card{\Is} + \card{\Is} + \card{V} $.
The push-relabel algorithm (see \cite{cheriyan1989analysis}), for example
solving this minimum-cost flow problem within $ \OO(\card{V}^2 \sqrt{\card{E}}) $.
Note that \cref{alg:sep-lift-sep} is super-exponential.
However, for fixed~$m$ it runs in polynomial time.

\section{Transformation between lower- and full-dimensional space}
\label{sec:traf_low_dim}

Recall that we have defined the multiple-choice set 
\begin{equation*}
	\X^\Is \coloneqq \SSet{\x \in \F^{\X}}{\sum_{i \in \I} \x_i \leq 1
		\zsetforall \I \in \Is}
\end{equation*}
and the boolean quadric polytope with multiple-choice constraints
\begin{equation*}
	\zP(G, \Is) \coloneqq \conv\SSet{(\x, \y, \z) \in \F^{\X \cup \Y \cup E}}
		{\x_i \y_\j = \z_{i \j}\, \forall \set{i, \j} \in E, x \in \X^\Is}
\end{equation*}
as full-dimensional polyhedra.

In the following, we give a one-to-one transformation between valid inequalities for~$ \zP(G, \Is) $
and valid inequalities for a lower-dimensional variant of~$ \zP(G, \Is) $,
where all multiple-constraints have to be fulfilled with equality.
Indeed, this lower-dimensional variant corresponds to a face of~$ \zP(G, \Is) $.

We assume that~$G$ is a complete bipartite graph.
Then we define a new complete bipartite graph $ \tilde{G} = (\tilde{X} \cup Y, \tilde{E}) $
with one extra node in each subset in the partition $ \Is $ of $X$.
Let $ \tilde{\Is} $ be the corresponding partition of $ \tilde{X} $.
Further, we define the multiple-choice set
\begin{equation}
	\tilde{\X}^\Is \coloneqq
		\SSet{\x \in \FR^{\tilde{\X}}}{\sum_{i \in \I} \x_i = 1
			\zsetforall \I \in \tilde{\Is}}
\label{X:full-dim}
\end{equation}
and the polytope
\begin{equation}
	\tilde{\zP}(\tilde{G}, \tilde{\Is}) \coloneqq
		\conv\SSet{(\x, \y, \z)\in \F^{\tilde{\X} \cup \Y \cup \tilde{E}}}
			{\x_i \y_\j = \z_{i \j}\, \forall \set{i, \j} \in \tilde{E},\, x \in \tilde{\X}^\Is}.
\end{equation}

We will now state an affine transformation which rotates $ \tilde{\X}^\Is $
in such a way that one variable per subset of the partition $ \tilde{\Is} $
becomes constantly zero.
For each $ I \in \tilde{\Is} $ we choose an arbitrary nodes $ i^I_0 \in I $
and define the matrices $ B_I \in \R^{I \times I} $,
$ B \in \R^{\tilde{X} \times \tilde{X}} $
and vectors $ b_I \in \R^I $, $ b \in \R^{\tilde{X}} $ as follows:
\begin{equation*}
	B_I(i_1, i_2) = \begin{cases}
		1 & \text{for } i_1 = i_2\\
		1 & \text{for } i_1 = i^I_0\\
		0 & \text{else }
		\end{cases}, \quad
	B = \begin{bmatrix}
		B_{I_1} &  & 0\\
		& \ddots &\\
		0 &    & B_{I_m}
		\end{bmatrix},
\end{equation*}
\begin{equation*}
	b_I(i) = \begin{cases}
		-1 & \text{for } i = i^I_0\\
		0 & \text{else }
		\end{cases}, \quad
	b = \begin{bmatrix}
		b_{I_1}\\
		\vdots\\
		b_{I_m}
		\end{bmatrix}.
\end{equation*}
By applying the invertible affine transformation $ f\colon \R^{\tilde{X}} \to \R^{\tilde{X}},
	x \mapsto Bx - b $
to $ \tilde{\X}^\Is $, we arrive at
\begin{equation*}
	\bar{\X}^\Is \coloneqq f(\tilde{\X}^I) =
		\SSet{\x \in \F^{\tilde{\X}}}{\sum_{i \in \I \setminus \set{i^I_0}} \x_i \leq 1,\,
			\x_{i^I_0} = 0 \zsetforall \I \in \tilde{\Is}}.
	\label{equ:qR}
\end{equation*}
This allows us to define the polytope
\begin{equation*}
	\bar{\zP}(\tilde{G}, \tilde{\Is}) \coloneqq
		\conv\SSet{(\x, \y, \z) \in \F^{\tilde{\X} \cup \Y \cup \tilde{E}}}
			{\x_i \y_\j = \z_{i \j}\, \forall \set{i, \j} \in \tilde{E},\, x \in \bar{\X}^\Is},
\end{equation*}
which is the canonical embedding of $ \zP(G, \Is) $ into $ \R^{\tilde{X^I}} = \R^{\bar{X^I}} $
(in fact an extended formulation).
Using Lemma~2 in \cite{galliseparable},
we can infer the two relations
\begin{equation*}
	\bar{\zP}(\tilde{G}, \tilde{\Is}) =
		\SSet{(Bx - b, y , Bz - by^T) \in \FR^{\tilde{\X} \cup \Y \cup \tilde{E}}}
			{(x, y, z) \in \tilde{\zP}(\tilde{G}, \tilde{\Is})}
\end{equation*}
and
\begin{equation*}
	\tilde{\zP}(\tilde{G}, \tilde{\Is}) =
		\SSet{(B^{-1}(x + b), y' ,B^{-1}(z + by^T)) \in \FR^{\tilde{\X} \cup \Y \cup \tilde{E}}}
			{(x, y, z) \in \bar{\zP}(\tilde{G}, \tilde{\Is})}.
\end{equation*}
Based on this observation,
we can formulate the following lemma, which transform inequalities
back and forth between the two polytopes.
\begin{lemma}
	Let $ a^T (x, y, z) \leq b $ be a valid inequality
	for $ \bar{\zP}(\tilde{G}, \tilde{\Is}) $,
	then the inequality $ a^T (B^{-1}(x + b), y, B^{-1}(z + by^T)) \leq \bar{b} $
	is valid for $ \tilde{\zP}(\tilde{G}, \tilde{\Is}) $ and vice versa.
	\label{lmm:trafo}
\end{lemma}

One can easily see that $ \bar{\zP}(\tilde{G}, \tilde{\Is}) $
and $ \zP(G, \Is) $ are the \qm{same} polytope in different dimension.
As all variables which do not appear in $ \zP(G, \Is) $
are fixed to $0$ in $ \bar{\zP}(\tilde{G}, \tilde{\Is}) $,
we can conclude from \cref{prop:facet-ext}
that the extension of a facet of $ \zP(G, \Is) $
is also a facet of $ \bar{\zP}(\tilde{G}, \tilde{\Is}) $.
In addition, the following \emph{basic equations} hold
for $ \bar{\zP}(\tilde{G}, \tilde{\Is}) $
(\cf the basic inequalities in \cref{sec:basic_rlt}):
\begin{alignat}{1}
	x_{i^I_0} = 0,\, \quad \forall \I \in \tilde{\Is},
		\label{Eq:Basic_Equations1} \\
	z_{i^I_0 j} = 0,\, \quad \forall \I \in \tilde{\Is}, \forall j \in Y. \label{Eq:Basic_Equations2}	
\end{alignat}

Via the inverse transformation to~$f$
applied to \cref{Eq:Basic_Equations2},
we obtain the following further valid equations:
\begin{equation}
	\sum_{i \in I} z_{ij} = y_j, \quad \forall j \in Y, I \in \hat{\Is}.
	\label{rlt:equations}
\end{equation}
As these equations can also be derived from~\cref{Eq:Basic_Equations2}
via the RLT procedure, we call them the \emph{RLT equations}.

Altogether, in order to obtain tighter linear relaxations
for $ \tilde{\zP}(\tilde{G}, \tilde{\Is}) $,
we can add the RLT equations to the initial formulation
and can derive further valid inequalities as follows.
We remove one arbitrary node from each subset $ \tilde{\Is} $,
which yields a graph $ \hat{G} = (\hat{X}, \hat{E}) $
and the corresponding partition $ \hat{\Is} $ on $ \hat{X} $.
Any valid inequality for $ P(\hat{G}, \hat{\Is}) $
can now directly be added to the relaxation as well
(as a $0$-lifted inequality).

In the pooling problem with recipes, which we have presented in \cref{comp:pooling},
we need to consider multiple instances of $ \tilde{\zP}(\tilde{G}, \tilde{\Is}) $
occurring as a substructure of the overall problem.
In order to make use of the above technique to improve the relaxation,
we first rescale all multiple-choice constraints~\cref{equ:qrecipe1}
such that they have a right-hand side of~$1$ instead of~$ \sigma_h $,
and then rescale all variables to have an upper bound of~$1$.
This way, the polytopes \cref{Eq:qyv-P}
are indeed of the form $ \tilde{\zP}(\tilde{G}, \tilde{\Is}) $,
which allows us to separate the cutting planes derived in \cref{sec:fac-def-in}
using the techniques presented in \cref{sec:sep}.

\section{Number of cutting planes produced in \cref{sec:res-random}}
\label{sec:num-cutting-planes}

In \cref{sec:res-random}, we conducted experiments on random instances
to find out which classes of facet-defining inequalities we derived
are able to close the integrality gap how far.
The results were given in \cref{tb:random_rel_reduction}.
Here we report how many cutting planes of each type were needed
for each given instance type to produce these results.
In \cref{tb:random_added_cuts},
we see the number of cutting planes found
for the corresponding cells of \cref{tb:random_rel_reduction},
again averaged over all~10 instances of each type.
\begin{table}[h]
	\centering
	\caption{Number of cutting planes separated
		for instances of different size and for various facet classes
		until no more violated inequalities were found}
	\label{tb:random_added_cuts}
	\begin{tabular}{ l|rrr|rrr|r }
		\toprule
		& 5-5-10 & 10-10-10 & 15-15-10 & 5-5-20 & 5-5-40 & 5-5-60 & 10-*-25\\
		\midrule
		RLT & \numprint{100} & \numprint{200} & \numprint{300} & \numprint{200} & \numprint{400} & \numprint{600} & \numprint{500}\\
		C & \numprint{460} & \numprint{1014} & \numprint{1180} & \numprint{1937} & \numprint{4963} & \numprint{13473} & \numprint{3796}\\
		CC & \numprint{361} & \numprint{1026} & \numprint{1709} & \numprint{2496} & \numprint{12242} & \numprint{26285} & \numprint{9749}\\
		A1 & \numprint{979} & \numprint{8390} & \numprint{30839} & \numprint{3048} & \numprint{5448} & \numprint{7662} & \numprint{16460}\\
		A1S & \numprint{1538} & \numprint{28946} & \numprint{141927} & \numprint{8682} & \numprint{23971} & \numprint{36671} & \numprint{57074}\\
		A1C & \numprint{418} & \numprint{3370} & \numprint{8927} & \numprint{1716} & \numprint{4036} & \numprint{6546} & \numprint{8448}\\
		A2 & \numprint{355} & \numprint{1076} & \numprint{1947} & \numprint{1134} & \numprint{3979} & \numprint{8428} & \numprint{5451}\\
		A2S & \numprint{1454} & \numprint{5886} & \numprint{12591} & \numprint{4704} & \numprint{14714} & \numprint{31140} & \numprint{22551}\\
		A2C & \numprint{646} & \numprint{4670} & \numprint{12008} & \numprint{2490} & \numprint{6682} & \numprint{12071} & \numprint{10674}\\
		All & \numprint{1954} & \numprint{20464} & \numprint{69273} & \numprint{8635} & \numprint{31591} & \numprint{62479} & \numprint{57206}\\
		\bottomrule
	\end{tabular}
\end{table}

The RLT inequalities, which had all been added from the start,
are a comparably small class of facet-defining inequalities.
Nevertheless, we saw from our experiments on both random and real-world instances,
that they are very helpful in moving the dual bound.

The number of CC inequalities we found is about 2.7~times
the number of C inequalities,
which is an indication that the copying operation
is able to increase the class of ordinary cycle facets tremendously.
It is also interesting to note
that much fewer A1C inequalities than original A1 inequalities are found,
but nevertheless the former provide a significantly better dual bound.
This and the results for the C/CC inequalities
point to a high potential of the coyping operation
as a lifting method for a given basic class of valid inequalities.
The All row finally shows that a relatively high number of violated inequalities
is separated when considering all classes jointly
and iterating until no further violated inequalities are found -- up to \numprint{70 000}
for the largest instances.

\section{Second strongest facet class after CC}
\label{sec:further-study}

From the results in \cref{tb:random_rel_reduction,tb:random_added_cuts},
we concluded that the CC inequalities are by far the most efficient ones to separate:
they provide a very strong improvement in the dual bound per putting plane
added to the relaxation, and at the same time,
their separation is possible at very low computational cost.
This motivated us to examine which of the facet classes we found
is the second-strongest after the CC facets.
In \cref{tb:random_second_step}, we show the results obtained
for the largest instances in our test set
when first adding all violated CC inequalities
and then iteratively separating the CC inequalities
and the respective second class of inequalities jointly
until no further violated inequalities are found.
\begin{table}[H]
	\centering
	\caption{Average optimality gaps in \%
		when separating the CC inequalities together with a second inequality class		
		for the three classes of instances from \cref{tb:random_rel_reduction}
		with a non-zero gap after separating the CC inequalities alone (left).
		The corresponding number of cutting planes found (right).}
	\label{tb:random_second_step}
	\begin{tabular}{ l|rrr||rrr }
		\toprule
		& 5-5-40 & 5-5-60 & 10-*-25 & 5-5-40 & 5-5-60 & 10-*-25 \\
		\midrule
		CC & 1.99 & 4.43 & 4.05 & \numprint{12242} & \numprint{26285} & \numprint{9749} \\
		\midrule
		RLT & 1.41 & 3.49 & 3.19 & \numprint{400}(+\numprint{335}) & \numprint{600}(+\numprint{827}) & \numprint{500}(+\numprint{645}) \\
		A1S & 0.88 & 2.19 & 2.54 & \numprint{2534}(+\numprint{451}) & \numprint{8920}(+\numprint{1307}) & \numprint{5367}(+\numprint{894}) \\
		A1C & 1.28 & 2.74 & 3.48 & \numprint{957}(+\numprint{344}) & \numprint{3426}(+\numprint{1069}) & \numprint{1150}(+\numprint{479}) \\
		A2S & 1.97 & 4.38 & 3.61 & \numprint{48}(+\numprint{34}) & \numprint{125}(+\numprint{82}) & \numprint{703}(+\numprint{383}) \\
		A2C & 0.70 & 1.83 & 2.17 & \numprint{1799}(+\numprint{660}) & \numprint{5657}(+\numprint{1928}) & \numprint{4061}(+\numprint{1314}) \\
		\midrule
		All & 0.45 & 1.48 & 1.64 &&&\\
		\bottomrule
	\end{tabular}
\end{table}
The left-hand side of the table shows the remaining optimality gap
while the right-hand side shows the number of cutting planes found.
In the CC row and the All row, we repeat the results
from \cref{tb:random_rel_reduction} for comparison.
The five rows in the middle indicate the results for choosing
RLT, A1S, A1C, A2S and A2C, respectively,
as the inequality class to separate jointly with CC.
Furthermore, we see how many additional cutting planes
of the respective second class were needed,
with the number of additional CC inequalities in parentheses.
We infer from the results that all classes make a certain further contribution
to closing the gap compared to separating the CC inequalities alone.
The A2C inequalities seem to be the most promising ones;
however, they require the repeated solution of an integer program to separate.
An almost equally good result is obtained for the A1S inequalities,
where the subproblem is only a continuous flow problem.
Finally, we see that the remaining gaps in the All row are significantly lower,
which means that using more than two facet classes still yields further progress.

%% file: preprint.bbl
\newcommand{\etalchar}[1]{$^{#1}$}
\begin{thebibliography}{GKRW20}

\bibitem[ABH{\etalchar{+}}04]{audet2004pooling}
Charles Audet, Jack Brimberg, Pierre Hansen, S{\'e}bastien~Le Digabel, and
  Nenad Mladenovi{\'c}.
\newblock Pooling problem: Alternate formulations and solution methods.
\newblock {\em Management science}, 50(6):761--776, 2004.

\bibitem[AI07]{avis2007new}
David Avis and Tsuyoshi Ito.
\newblock New classes of facets of the cut polytope and tightness of $ i_{mm22}
  $ {B}ell inequalities.
\newblock {\em Discrete Applied Mathematics}, 155(13):1689--1699, 2007.

\bibitem[AII08]{avis2008generating}
David Avis, Hiroshi Imai, and Tsuyoshi Ito.
\newblock Generating facets for the cut polytope of a graph by triangular
  elimination.
\newblock {\em Mathematical programming}, 112(2):303--325, 2008.

\bibitem[AIIS05]{avis2005two}
David Avis, Hiroshi Imai, Tsuyoshi Ito, and Yuuya Sasaki.
\newblock Two-party {B}ell inequalities derived from combinatorics via
  triangular elimination.
\newblock {\em Journal of Physics A: Mathematical and General},
  38(50):10971--10987, 2005.

\bibitem[BDK{\etalchar{+}}17]{BolandDeyKalinowskiMolinaroRigterink2015}
Natashia Boland, Santuna~S. Dey, Thomas Kalinowski, Marco Molinaro, and Fabian
  Rigterink.
\newblock Bounding the gap between the {M}c{C}ormick relaxation and the convex
  hull for bilinear functions.
\newblock {\em Mathematical Programming, Series A}, 162:523--535, 2017.

\bibitem[BGM20]{cycle-free2020}
Andreas Bärmann, Patrick Gemander, and Maximilian Merkert.
\newblock The clique problem with multiple-choice constraints under a
  cycle-free dependency graph.
\newblock {\em Discrete Applied Mathematics}, 283:59--77, 2020.

\bibitem[BGMS18]{staircase2018}
Andreas Bärmann, Thorsten Gellermann, Maximilian Merkert, and Oskar Schneider.
\newblock Staircase compatibility and its applications in scheduling and
  piecewise linearization.
\newblock {\em Discrete Optimization}, 29:111--132, 2018.

\bibitem[BH93]{boros1993cut}
Endre Boros and Peter~L Hammer.
\newblock Cut-polytopes, boolean quadric polytopes and nonnegative quadratic
  pseudo-boolean functions.
\newblock {\em Mathematics of Operations Research}, 18(1):245--253, 1993.

\bibitem[BM86]{barahona1986cut}
Francisco Barahona and Ali~Ridha Mahjoub.
\newblock On the cut polytope.
\newblock {\em Mathematical programming}, 36(2):157--173, 1986.

\bibitem[BMS20]{benders2020}
Andreas Bärmann, Alexander Martin, and Oskar Schneider.
\newblock Efficient formulations and decomposition approaches for power peak
  reduction in railway traffic via timetabling.
\newblock {\em Transportation Science}, 2020.
\newblock To appear.

\bibitem[BZ04]{bienstock2004subset}
Daniel Bienstock and Mark Zuckerberg.
\newblock Subset algebra lift operators for 0-1 integer programming.
\newblock {\em SIAM Journal on Optimization}, 15(1):63--95, 2004.

\bibitem[Cas15]{Castro2015}
Pedro~M. Castro.
\newblock Tightening piecewise {M}c{C}ormick relaxations for bilinear problems.
\newblock {\em Computers \& Chemical Engineering}, 72:300--311, 2015.

\bibitem[CG04]{collins2004relevant}
Daniel Collins and Nicolas Gisin.
\newblock A relevant two qubit {B}ell inequality inequivalent to the {CHSH}
  inequality.
\newblock {\em Journal of Physics A: Mathematical and General},
  37(5):1775--1787, 2004.

\bibitem[CM89]{cheriyan1989analysis}
Joseph Cheriyan and SN~Maheshwari.
\newblock Analysis of preflow push algorithms for maximum network flow.
\newblock {\em SIAM Journal on Computing}, 18(6):1057--1086, 1989.

\bibitem[CSPB17]{CusticSokolPunnenBhattacharya2017}
Ante \'{C}usti\'{c}, Vladyslav Sokol, Abraham~P. Punnen, and Binay
  Bhattacharya.
\newblock The bilinear assignment problem: Complexity and polynomially solvable
  special cases.
\newblock {\em Mathematical Programming, Series A}, 166:185--205, 2017.

\bibitem[DLL11]{d2011valid}
Claudia D’Ambrosio, Jeff Linderoth, and James Luedtke.
\newblock Valid inequalities for the pooling problem with binary variables.
\newblock In {\em International Conference on Integer Programming and
  Combinatorial Optimization}, pages 117--129. Springer, 2011.

\bibitem[DLW97]{deza1997geometry}
MM~Deza, Monique Laurent, and R~Weismantel.
\newblock Geometry of cuts and metrics.
\newblock {\em Mathematical Methods of Operations Research-ZOR},
  46(3):282--283, 1997.

\bibitem[DS90]{de1990cut}
Caterina De~Simone.
\newblock The cut polytope and the boolean quadric polytope.
\newblock {\em Discrete Mathematics}, 79(1):71--75, 1990.

\bibitem[FL18]{fampa2018efficient}
Marcia Fampa and Jon Lee.
\newblock Efficient treatment of bilinear forms in global optimization.
\newblock \url{https://arxiv.org/abs/1803.07625}, 2018.

\bibitem[FT05]{faye2005polyhedral}
Alain Faye and Quoc-an Trinh.
\newblock A polyhedral approach for a constrained quadratic 0--1 problem.
\newblock {\em Discrete Applied Mathematics}, 149(1-3):87--100, 2005.

\bibitem[GACD13]{gupte2013solving}
Akshay Gupte, Shabbir Ahmed, Myun~Seok Cheon, and Santanu Dey.
\newblock Solving mixed integer bilinear problems using {MILP} formulations.
\newblock {\em SIAM Journal on Optimization}, 23(2):721--744, 2013.

\bibitem[GADC17]{gupte2017relaxations}
Akshay Gupte, Shabbir Ahmed, Santanu~S Dey, and Myun~Seok Cheon.
\newblock Relaxations and discretizations for the pooling problem.
\newblock {\em Journal of Global Optimization}, 67(3):631--669, 2017.

\bibitem[GGL19]{galliseparable}
Laura Galli, Akshay Gupte, and Adam~N Letchford.
\newblock Separable bilinear programs, facial disjunctions and the
  reformulation-linearization technique.
\newblock
  \url{https://pdfs.semanticscholar.org/8adc/5b08fdfbb4f2fdfa1c7584ff78810aeb3567.pdf},
  2019.

\bibitem[GJ00]{polymake:2000}
Ewgenij Gawrilow and Michael Joswig.
\newblock {\tt polymake}: a framework for analyzing convex polytopes.
\newblock In {\em Polytopes---combinatorics and computation ({O}berwolfach,
  1997)}, volume~29 of {\em DMV Sem.}, pages 43--73. Birkh\"auser, Basel, 2000.

\bibitem[GKRW20]{gupte2020extended}
Akshay Gupte, Thomas Kalinowski, Fabian Rigterink, and Hamish Waterer.
\newblock Extended formulations for convex hulls of some bilinear functions.
\newblock {\em Discrete Optimization}, 36:100569, 2020.

\bibitem[GLL12]{gunluk2012polytope}
Oktay G{\"u}nl{\"u}k, Jon Lee, and Janny Leung.
\newblock A polytope for a product of real linear functions in 0/1 variables.
\newblock In {\em Mixed Integer Nonlinear Programming}, pages 513--529.
  Springer, 2012.

\bibitem[Gup16]{gupteworking}
Akshay Gupte.
\newblock A note on simplicial bilinear optimization.
\newblock \url{http://agupte.people.clemson.edu/BilinSimpl.pdf}, 2016.

\bibitem[{Gur}20]{gurobi}
{Gurobi Optimization, LLC}.
\newblock Gurobi optimizer reference manual.
\newblock \url{http://www.gurobi.com}, 2020.

\bibitem[HLL98]{HardinLeeLeung1998}
Jill Hardin, Jon Lee, and Janny Leung.
\newblock On the boolean-quadric packing uncapacitated facility location
  polytope.
\newblock {\em Annals of Operations Research}, 83:77--94, 1998.

\bibitem[KCG13]{KolodziejCastroGrossmann2013}
Scott Kolodziej, Pedro~M. Castro, and Ignacio~E. Grossmann.
\newblock Global optimization of bilinear programs with a multiparametric
  disaggregation technique.
\newblock {\em Journal of Global Optimization}, 57:1039--1063, 2013.

\bibitem[KPP04]{KellererPferschyPisinger2004}
Hans Kellerer, Urlrich Pferschy, and David Pisinger.
\newblock {\em Knapsack Problems}, chapter The Multiple-Choice Knapsack
  Problem, pages 317--347.
\newblock Springer, 2004.

\bibitem[LG17]{LimaGrossmann2017}
Ricardo~M. Lima and Ignacio~E. Grossmann.
\newblock On the solution of nonconvex cardinality boolean quadratic
  programming problems: A computational study.
\newblock {\em Computational Optimization and Applications}, 66(1):1--37, 2017.

\bibitem[LL04]{lee2004boolean}
Jon Lee and Janny Leung.
\newblock On the boolean quadric forest polytope.
\newblock {\em INFOR: Information Systems and Operational Research},
  42(2):125--141, 2004.

\bibitem[LM16]{liers2016structural}
Frauke Liers and Maximilian Merkert.
\newblock Structural investigation of piecewise linearized network flow
  problems.
\newblock {\em SIAM Journal on Optimization}, 26(4):2863--2886, 2016.

\bibitem[LS14]{letchford2014new}
Adam~N Letchford and Michael~M S{\o}rensen.
\newblock A new separation algorithm for the boolean quadric and cut polytopes.
\newblock {\em Discrete Optimization}, 14:61--71, 2014.

\bibitem[McC76]{mccormick1976computability}
Garth~P McCormick.
\newblock Computability of global solutions to factorable nonconvex programs:
  Part {I} -- convex underestimating problems.
\newblock {\em Mathematical programming}, 10(1):147--175, 1976.

\bibitem[Meh97]{mehrotra1997cardinality}
Anuj Mehrotra.
\newblock Cardinality constrained boolean quadratic polytope.
\newblock {\em Discrete Applied Mathematics}, 79(1-3):137--154, 1997.

\bibitem[MF09]{misener2009advances}
Ruth Misener and Christodoulos~A Floudas.
\newblock Advances for the pooling problem: modeling, global optimization, and
  computational studies.
\newblock {\em Appl. Comput. Math}, 8(1):3--22, 2009.

\bibitem[Nau87]{Nauss1978}
Robert~M. Nauss.
\newblock The 0-1 knapsack problem with multiple choice constraints.
\newblock {\em European Journal of Operational Research}, 2(2):125--131, 1987.

\bibitem[Pad89]{padberg1989boolean}
Manfred Padberg.
\newblock The boolean quadric polytope: some characteristics, facets and
  relatives.
\newblock {\em Mathematical programming}, 45(1-3):139--172, 1989.

\bibitem[Pen07]{pentico2007assignment}
David~W Pentico.
\newblock Assignment problems: A golden anniversary survey.
\newblock {\em European Journal of Operational Research}, 176(2):774--793,
  2007.

\bibitem[Pit91]{pitowsky1991correlation}
Itamar Pitowsky.
\newblock Correlation polytopes: their geometry and complexity.
\newblock {\em Mathematical Programming}, 50(1-3):395--414, 1991.

\bibitem[PSK13]{PunnenSripratakKarapetyan2013}
Abraham~P. Punnen, Piyashat Sripratak, and Daniel Karapetyan.
\newblock Domination analysis of algorithms for bipartite boolean quadratic
  programs.
\newblock In {\em International Symposium on Fundamentals of Computation Theory
  (FCT)}, pages 271--282, 2013.

\bibitem[PSK15]{PunnenSripratakKarapetyan2015}
Abraham~P. Punnen, Piyashat Sripratak, and Daniel Karapetyan.
\newblock The bipartite unconstrained 0-1 quadratic programming problem:
  Polynomially solvable cases.
\newblock {\em Discrete Applied Mathematics}, 193:1--10, 2015.

\bibitem[PW16]{PunnenWang2016}
Abraham~P. Punnen and Yang Wang.
\newblock The bipartite quadratic assignment problem and extensions.
\newblock {\em Discrete Optimization}, 250:715--725, 2016.

\bibitem[SA92]{sherali1992new}
Hanif~D Sherali and Amine Alameddine.
\newblock A new reformulation-linearization technique for bilinear programming
  problems.
\newblock {\em Journal of Global optimization}, 2(4):379--410, 1992.

\bibitem[SLA95]{sherali1995simultaneous}
Hanif~D Sherali, Youngho Lee, and Warren~P Adams.
\newblock A simultaneous lifting strategy for identifying new classes of facets
  for the boolean quadric polytope.
\newblock {\em Operations Research Letters}, 17(1):19--26, 1995.

\bibitem[SPS19]{sripratak16bipartite}
Piyashat Sripratak, Abraham~P. Punnen, and Tamon Stephen.
\newblock The bipartite boolean quadric polytope.
\newblock Technical report, Simon Fraser University, 2019.

\bibitem[Sri14]{sripratak2014bipartite}
Piyashat Sripratak.
\newblock {\em The bipartite boolean quadratic programming problem}.
\newblock PhD thesis, Science: Mathematics, 2014.

\bibitem[WW01]{werner2001bell}
Reinhard~F. Werner and Michael~M. Wolf.
\newblock Bell inequalities and entanglement.
\newblock \url{https://arxiv.org/abs/quant-ph/0107093}, 2001.

\bibitem[Zuc16]{zuckerberg2016geometric}
Mark Zuckerberg.
\newblock Geometric proofs for convex hull defining formulations.
\newblock {\em Operations Research Letters}, 44(5):625--629, 2016.

\end{thebibliography}
